\documentclass[11pt,a4paper,reqno,tbtags]{amsart}
\usepackage{amsmath}
\usepackage{amsfonts}
\usepackage{amssymb}
\usepackage{bbm}
\usepackage{latexsym}
\usepackage{graphicx}
\usepackage[utf8]{inputenc}
\usepackage{mathrsfs}
\usepackage{epstopdf}
\usepackage{hyperref}
\usepackage{url}
\usepackage{mdwlist}

\numberwithin{equation}{section}

\newtheorem{proposition}{Proposition}[section]
\newtheorem{corollary}[proposition]{Corollary}
\newtheorem{theorem}[proposition]{Theorem}

\newtheorem{lemma}[proposition]{Lemma}

\theoremstyle{definition}

\newtheorem*{ack}{Acknowledgement}

\theoremstyle{remark}

\newcommand{\open}{\mathcal{R}}
\newcommand{\mean}{\mu}
\newcommand{\off}{\xi}
\newcommand{\offhat}{\hat{\off}}
\newcommand{\sgt}{\mathcal{T}}
\newcommand{\sgloop}{\mathcal{L}}
\renewcommand{\root}{\varnothing}
\newcommand{\Loop}{\mathrm{Loop}}
\newcommand{\Tset}{\mathfrak{T}}
\newcommand{\Lset}{\mathfrak{L}}
\newcommand{\Tree}{\tau}
\renewcommand\P{\mathbb{P}}

\newcommand\E{\mathbb{E}}

\renewcommand{\l}{\lambda}

\newcommand{\specdim}{d_{\mathrm{s}}}

\newcommand{\indic}[1]{\mathbf{1}_{\{#1\}}}

\newcommand\RR{\mathbb{R}}
\newcommand\Reff{\mathrm{R}_{\mathrm{eff}}}

\newcommand{\bea}{\begin{eqnarray}}
\newcommand{\eea}{\end{eqnarray}}

\def\void{}
\def\labelmark{}

\newenvironment{formula}[1]{\def\labelname{#1}
\ifx\void\labelname\def\junk{\begin{displaymath}}
\else\def\junk{\begin{equation}\label{\labelname}}\fi\junk}%
{\ifx\void\labelname\def\junk{\end{displaymath}}
\else\def\junk{\end{equation}}\fi\junk\labelmark\def\labelname{}}

{\ifx\void\labelname\def\junk{\end{array}\end{displaymath}}
\else\def\junk{\end{array}\right.\end{equation}}
\fi\junk\labelmark\def\labelname{}\def\junk{}
}

\newcommand{\beq}{\begin{formula}}
\newcommand{\eeq}{\end{formula}}
\newcommand{\beqv}{\begin{formula}{}}

\newenvironment{romenumerate}[1][0pt]{% optional argument changes indentation
\addtolength{\leftmargini}{#1}\begin{enumerate}% gives (i), (ii) etc.
 }{\end{enumerate}}

\newcounter{oldenumi}
{\setcounter{oldenumi}{\value{enumi}}
\begin{romenumerate} \setcounter{enumi}{\value{oldenumi}}}
{\end{romenumerate}}

\begingroup
  \divide\count255 by 60
  \multiply\count255 by -60
  \advance\count255 by \time
  \ifnum \count255 < 10 \xdef\klockan{\the\count1.0\the\count255}
  \else\xdef\klockan{\the\count1.\the\count255}\fi
\endgroup

\newcommand\marginal[1]{\marginpar{\raggedright\parindent=0pt\tiny #1}}
\newcommand\REM[1]{{\raggedright\texttt{[#1]}\par\marginal{XXX}}}

\newcommand\eps{\varepsilon}

\def\rompar(#1){\textup(#1\textup)}    % usage: \rompar(...)

\def\xexp(#1){e^{#1}}
\newcommand\ceil[1]{\lceil#1\rceil}
\newcommand\floor[1]{\lfloor#1\rfloor}

\newcommand{\tend}{\longrightarrow}
\newcommand\dto{\overset{\mathrm{d}}{\tend}}

\newcommand\ZZ{\mathbb Z}

\newcommand{\cC}{\mathcal{C}}

\newcounter{CC}
\newcounter{cc}
\newcommand{\out}{\text{out}}
\newcommand{\tree}{\tau}
\newcommand{\height}{\mathrm{Height}}

\allowdisplaybreaks

\begin{document}
\title%[]
{Random walk on random infinite looptrees}

\date{\today}   %\Small
%05C80, 05C05, 60J80, 60F05} 
%; revised}%Random trees, simply generated trees, branching process, weak limit.}%; revised ...
% Pacs: 05.40.Fb Random walks and Levy flights, 05.50.+q Lattice theory and statistics (Ising, Potts, etc.) 04.60.Nc Lattice and discrete methods

\author{Jakob E. Bj\"ornberg}

\address{Department of
Mathematical Sciences, Chalmers and
Gothenburg University, 412 96 G\"oteborg, Sweden.
Email: jakob.bjornberg$@$gmail.com}

\author{Sigurdur \"Orn Stef\'ansson} %%  Sigurður %
\address{
Division of Mathematics, The
Science Institute, University of Iceland,
Dunhaga 3 IS-107 Reykjavik, Iceland.
Email: rudrugis$@$gmail.com}

\thanks{Research supported by the Knut and Alice Wallenberg Foundation.}

\maketitle

\begin{abstract} 
Looptrees have recently arisen in the study of critical percolation on
the uniform infinite planar triangulation.  Here we consider random infinite
looptrees defined as the local limit of the looptree associated with a
critical Galton--Watson tree conditioned to be large.  We study simple
random walk on these infinite looptrees by means of providing
estimates on volume and resistance growth.  We prove that if the
offspring distribution of the Galton--Watson process is in the domain
of attraction of a stable distribution with index $\alpha\in(1,2]$
then the spectral dimension of the looptree is $2\alpha/(\alpha+1)$.
\end{abstract}

\section{Introduction}

 Random graphs have been used as a discretization of continuous
 manifolds in a `sum over histories' approach to quantum gravity
 \cite{durhuus:book}. In physics these graphs go by the name of
 \emph{dynamical triangulations} since the basic building blocks of
 the graphs are triangles (if the manifold is two-dimensional) and
 their higher dimensional analogs.  It is not important for large
 scale properties how one chooses these building blocks and in two
 dimensions one may e.g.~replace triangles by polygons of higher
 degree and the general graph is in this case called a \emph{planar
   map}.

An important observable one would like to have in such theories is
some notion of dimension of the random graphs. Although planar maps
are locally two-dimensional in the sense that their building blocks
are polygons they may on large scales be far from two-dimensional. As
an example, consider the planar map with vertex set
$\{0,1,\ldots,n\}\times\{0,1\}$ and edges between $(i,j)$ and
$(i',j')$ iff $|i-i'| +|j-j'| = 1$. The graph consists of $n$ squares
connected together in a linear chain and when $n$ is large it is in
some ways more one-dimensional than two-dimensional. There are
different ways to define a dimension of a random graph and one way is
through the behaviour of simple random walk on the graph. The
\emph{spectral dimension} $\specdim(G)$ may be defined for any
connected, locally finite graph $G$ as the limit
\begin{equation}
\specdim(G)=-2\lim_{n\to\infty}\frac{\log p_{2n}^G(x,x)}{\log n}
\end{equation}
provided this exists
 ($\specdim(G)$ is then easily seen to be independent of $x$).  
Here $p_{2n}^G(x,x)$ is the probability that
simple random walk on $G$, started at $x$, returns to $x$ after $2n$
steps. The
spectral dimension of the $d$-dimensional lattice $\mathbb{Z}^d$
equals $d$ which motivates its definition.

The last decades have seen substantial progress on the understanding of
random walk on random graphs.  
Much of this work has been motivated by the Alexander--Orbach
conjecture~\cite{alexander:1982}, which concerns
random walk on critical percolation clusters in $\ZZ^d$, which
are supposed to be `tree-like'.  In a seminal 1986
paper~\cite{kesten:1986}, Kesten initiated a rigorous study of 
random walk on critical percolation clusters and 
large critical random trees.
More recently Barlow and
Kumagai~\cite{barlow:2006} studied simple random 
walk on the incipient infinite
cluster of percolation on the $d$-regular tree, which we denote here by $\sgt$.  
Among other things, Barlow and Kumagai proved that $\specdim(\mathcal{T})=4/3$ for almost every realization of
$\mathcal{T}$.  

The incipient infinite cluster $\mathcal{T}$
of the $d$-regular tree may be constructed as the large $N$ local
limit of a Galton--Watson tree with critical offspring 
distribution Bin($d$,$1/d$), conditioned to have size $N$.
Recall that a Galton--Watson tree 
with offspring distribution $\pi =(\pi_i)_{i\geq0}$
is a random tree
constructed recursively by starting with a single individual who gives
birth to a number of offspring according to the distribution $\pi$ and
then each of the offspring has independently offspring of their own
according to the same distribution and so on, see
e.g.~\cite{athreya-ney}.  The offspring distribution and the process are said to be \emph{critical} if the
expected number of offspring is 1.  We 
let $\sgt_N$ denote a Galton--Watson tree conditioned on having $N$
vertices.

Fuji and Kumagai~\cite{fuji:2008} generalized the result
of~\cite{barlow:2006} by showing that the same spectral dimension
$4/3$ is obtained for \emph{any} tree $\mathcal{T}$ arising as the
local limit of $\sgt_N$ when the offspring distribution $\pi$ is critical
and has finite variance.  Croydon and Kumagai~\cite{croydon:2008}
studied the case when the offspring distribution may have infinite
variance: more precisely they considered the case when $\pi$ is in the
domain of attraction of a stable distribution of index
$\alpha\in(1,2]$ which is equivalent to
\begin {equation} \label{stableintro}
 \sum_{i=0}^n i^2 \pi_i = n^{2-\alpha} L_1(n)
\end {equation}
where $L_1$ is slowly varying at infinity (see Section \ref{spec_sec}). 
They showed that then the spectral
dimension equals 
\begin {equation} \label{ck_ds}
d_\mathrm{s}(\sgt) = \frac{2\alpha}{2\alpha-1},
\quad\mbox{almost surely.}
\end {equation}
This includes the result
by Fuji and Kumagai for $\alpha=2$. 

The methods of the above-mentioned papers relied on connections
between the asymptotics of the random walk with volume and resistance
growth in the graph.  A general formulation of these methods, which
applies not only to trees but in principle to any random, strongly
recurrent graphs, was given by Kumagai and Misumi
in~\cite{kumagai:2008}. The general scheme of~\cite{kumagai:2008} is to show that
for random infinite graphs in which the volume of graph balls of radius $r$ grows roughly like $r^a$ and in which the resistance from the center to the boundary of the balls grows roughly like $r^b$,
the spectral dimension is 
\begin {equation} \label{dskm}
d_\mathrm{s} = 2a/(a+b), 
\end {equation}
 cf.~\cite[Eq.~(1.1)]{kumagai:2008}. The results of the current paper
rely on the general methods of~\cite{kumagai:2008}.
Other applications of these methods include determining the spectral
dimension of the local limit of random bipartite planar maps in a
certain `condensation phase', see~\cite{bjornberg:2013}. In parallel
to the above-mentioned progress, mathematical physicists have developed
generating function methods to calculate the spectral dimension of
random trees, see
e.g.~\cite{durhuus:2006,durhuus:2007,jonsson:2008,jonsson:2011,stefansson:2012}. The
benefits of those methods are their simplicity but the disadvantages
are that they do not apply as generally and give somewhat weaker
results regarding the existence of $d_\mathrm{s}$.

Recently a new tree-like random structure called a \emph{looptree} was
introduced by Curien and Kortchemski \cite{curien:2013a}.  Given a
finite rooted tree $\tree$ embedded in the plane, one may informally
define the corresponding rooted looptree $\Loop(\tree)$ as follows
(more detail is given
in Section~\ref{trees_sec}).  The
vertex set of $\Loop(\tree)$ is the same as that of $\tree$, but
instead of connecting each vertex to all its children, edges connect
\emph{siblings} in the order they appear in the embedding, and also
there are edges from the leftmost and rightmost children to the
parent.  This is illustrated in Fig.~\ref{f:introloop}.
(What we call $\Loop(\tree)$  is
called $\Loop(\tree)$ in~\cite{curien:2013b} and 
$\Loop'(\tree)$ in~\cite{curien:2013a}.)

\begin{figure} [ht]
\centerline{\scalebox{0.9}{\includegraphics{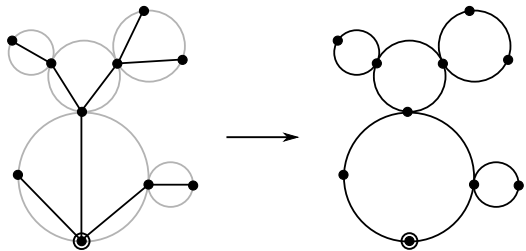}}}
\caption{On the left is a tree $\tree$ in black and the edges in its
  associated looptree in gray.  On the right is
  $\Loop(\tree)$.   The genealogical order in $\tree$ is
  defined with respect to the root which is
  circled.} \label{f:introloop}
\end{figure}

The recent paper~\cite{curien:2013a} studied the \emph{scaling limit}
of $\Loop(\sgt_N)$, where $\sgt_N$ is a critical Galton--Watson tree
conditioned to have size $N$ and whose offspring distribution
$\pi=(\pi_i)_{i\geq 0}$ satisfies \eqref{stableintro}. By scaling
limit we mean convergence in the Gromov--Hausdorff sense of a sequence
of graphs with a properly renormalized graph metric. For each
$\alpha$, the scaling limit was described directly in terms of an
$\alpha$--stable L\'evy process and was referred to as a \emph{stable
  looptree with index $\alpha$}. Furthermore it was proved that the
stable looptrees arise as the scaling limits of certain random
dissections of a polygon.  In~\cite{curien:2013b} it was shown that
the scaling limit of the boundary of a critical percolation cluster in
the uniform infinite planar triangulation equals the stable looptree
with index $\alpha =3/2$ and it was conjectured that the family of
stable looptrees appears generically in the study of interfaces in
statistical mechanical models on planar maps. Therefore, it is of
interest to understand the properties of stable looptrees in detail.

\subsection{Main results}

In this paper we study the \emph{local} limit $\sgloop$ of
$\Loop(\sgt_N)$ as $N\to\infty$ where $\sgt_N$ is a random tree which
belongs to the family of so-called simply generated random trees.  In
most cases of interest one may view $\sgt_N$ as a critical or
sub-critical Galton--Watson tree conditioned on having $N$ vertices as
we will explain in Section \ref{local_sec}. Local convergence of a
sequence of rooted graphs means that for any finite $R$, the graph
ball of radius $R$ around the root is eventually constant.  The
existence and construction of the local limit of $\Loop(\sgt_N)$ is
straightforward, given a recent general result on local limits of
simply generated trees by Janson~\cite{janson:2012sgt}, see Theorem
\ref{th:janson} and Corollary~\ref{sglopp_cor}.

After providing the general construction of
$\sgloop$ we will mainly focus on the case when
$\sgt_N$ is a critical size conditioned Galton--Watson tree whose
offspring distribution $\pi$ is in the domain of a attraction of a
stable law with index $\alpha \in (1,2]$, i.e.~satisfies
  \eqref{stableintro}. By using the results of~\cite{kumagai:2008},
  and establishing suitable bounds on volume and resistance growth in
  $\sgloop$, we then prove the following:
\begin{theorem}\label{specdim_thm}
If $\pi$ is critical and satisfies~\eqref{stableintro} then
for almost every realization of $\sgloop$, the spectral dimension
satisfies
\[
\specdim(\sgloop)=\frac{2\alpha}{\alpha+1}.
\]
\end{theorem}
Further results on the escape time from the graph ball of radius $R$
around the root will be stated
in Theorem~\ref{KM_thm}. Our strategy for proving Theorem~\ref{specdim_thm}
is to establish the volume growth exponent
$a=\alpha$ and the resistance growth exponent 
$b=1$, cf.~\eqref{dskm}.  We note that the resistance growth exponent is in fact the same as in the critical tree $\sgt$. An intuitive explanation for this resistance growth is that $\sgt$ consists of exactly one infinite simple path from the root  (a \emph{spine}) which has finite independent critical Galton-Watson outgrowths. The height of the outgrowths is sufficiently small so that only $O(1)$ number of vertices at level $r/2$ in $\sgt$ are connected to level $r$ in $\sgt$ via a path which does note intersect the spine. We will show that roughly the same picture applies to the looptrees.
%In interpreting this formula one can think of
%$\alpha$ as the volume exponent and 1 as the 
%resistance exponent, cf.~\cite[Eq.~(1.1)]{kumagai:2008}.

Theorem \ref{specdim_thm} is a statement about the spectral dimension
which holds for almost every realization of the random infinite
looptrees and in that case $d_\mathrm{s}$ is often referred to as the
{\it quenched spectral dimension}. Another way to view a quantitative
property of a random walk on a random graph is to average it over all
the realizations of the graphs. In that way, we define the so-called
{\it annealed spectral dimension} of the random looptrees as
\begin {equation}
 \bar{d}_{\mathrm{s}} =-2\lim_{n\to\infty}\frac{\log \E(p_{2n}^{\sgloop}(x,x))}{\log n}.
\end {equation}
In general, the quenched and annealed spectral dimensions need 
not be equal and one can exist while the other is not defined. 
We will show:
\begin{theorem}\label{specdim_thm_annealed}
If $\pi$ is critical and satisfies~\eqref{stableintro} then
the annealed spectral dimension of the infinite random 
looptree $\sgloop$ is
\[
\bar{d}_{\mathrm{s}}=\frac{2\alpha}{\alpha+1}.
\]
\end{theorem}

The annealed result relies on the same type of resistance and volume
estimates as the quenched result but requires stronger
bounds. Furthermore, one needs bounds on the expected volume of the
graph ball $B(n;\sgloop)$ of radius $n$ centered on the root in
$\sgloop$. This bound is provided in Lemma \ref{l:hausdorff} which
also shows that the so-called annealed 
volume growth exponent
%Hausdorff (or fractal) dimension, $\bar{d}_{\mathrm{H}}$, 
of $\sgloop$ is
\begin {equation} \label{hausdim}
% \bar{d}_{\mathrm{H}} := 
\lim_{n\rightarrow\infty}
 \frac{\log{\E(|B(n;\sgloop)|)}}{\log n} = \alpha.
\end {equation}
This value coincides with the 
(quenched) Hausdorff dimension of 
random stable looptrees, cf.~\cite[Thm.~1]{curien:2013a}.
We remark that for the infinite tree $\sgt$ with offspring
distribution \eqref{stableintro} studied in \cite{croydon:2008} it was
not possible to obtain annealed results corresponding to Theorem
\ref{specdim_thm_annealed} and Eq.~\eqref{hausdim} when the variance of
the offspring distribution is infinite. The reason is that the
vertices in the trees tend to have very large degrees which causes
trouble when averaging over the trees, e.g. $\E(|B(n;\sgt)|) =
\infty$. The degrees of vertices in the looptrees are however bounded
by 4 which eliminates this problem.

\subsection{Outline and notation}

The rest of this paper is organized as follows.  In
Section~\ref{trees_sec} we precisely define our model and present the
local limit theorem along with some basic estimates on looptrees.  
Theorems~\ref{specdim_thm} and \ref{specdim_thm_annealed} are 
proved in Section~\ref{spec_sec}.
In the Appendix we collect some results about slowly varying functions
and their application to some of the random variables which we study
in this paper.

The following notation will be used throughout the paper.
$\E$ denotes expected value and $\P$ denotes probability,
$\lfloor \cdot\rfloor$ and $\ceil{\cdot}$ denote the
floor and ceiling functions, % largest integer $\leq x$,
the minimum of $a$ and $b$ is written $a\wedge b$ and
the maximum $a\vee b$,
the indicator of an event $A$ is written $\mathbf{1}_A$, and
$f(x)\sim g(x)$ denotes that $f(x)/g(x)\to 1$
for the appropriate limiting value of $x$
(usually $x\to\infty$ or $x\to1^-$).

\section{Plane trees and looptrees}
\label{trees_sec}

In this paper we will define rooted plane trees in the same, slightly
unconventional, way as in \cite{bjornberg:2013}. Let $T_\infty$ be an
infinite tree with vertex set 
$V(T_\infty) = \bigcup_{n\geq0} \mathbb{Z}^n$
i.e.~the set of all finite sequences of integers. The unique element
in $\mathbb{Z}^0$ is called the root and is denoted by $\root$. The
concatenation of two elements $u,v\in V(T_\infty)$ is denoted by
$uv$. The edges in $T_\infty$ are defined by connecting each vertex
$vi$, $i\in \mathbb{Z}$, $v\in V(T_\infty)$ to the vertex $v$ by an
edge and we say that $v$ is the parent of $vi$ and conversely that
$vi$ is a child of $v$.  When vertices share the same parent $v$ we
call them siblings and they are ordered from left to right by
declaring that $vi$ is to the left of $vj$ if
\begin {itemize}
\item $i=0$ ($v0$ is the leftmost child) or
\item $ij > 0$ and $i<j$ or
\item $i > 0$ and $j<0$.
 \end {itemize} 
Two siblings are said to be adjacent if one is immediately to 
the left of the other, i.e.~if they are of the form $vi$ 
and $v(i-1)$, with $i\neq 0$.
We say that $v$ is a descendant of $u$,
written $v\succeq u$, if there is a self-avoiding path
from $\root$ to $v$ containing $u$.  If $v\succeq u$
and $v\neq u$ we write $v\succ u$.

A rooted plane tree $\Tree$ is defined as a subtree of $T_\infty$ which contains the root $\root$ and has the property that for every vertex $v\in V(\Tree)$ there is a number
$\out(v) \in \{0,1,\ldots\}\cup\{\infty\}$, called the \emph{outdegree} of
$v$, such that $vi \in V(\Tree)$ if and only if $\lfloor -\out(v)/2\rfloor
< i \leq \lfloor \out(v)/2\rfloor$. 
The child $vi$ of $v$ in $\Tree$ is said to be the
\emph{rightmost} child of $v$ if it has no siblings in $\Tree$
to its right (thus the rightmost child is $v0$ if 
$\out(v)=1$, $v1$ if $\out(v)=2$, or $v(-1)$ if $\out(v)\geq3$).
Siblings $vi$, $vj$ of $v$ in $\Tree$ are said to be \emph{adjacent} 
if they are either
adjacent in $T_\infty$, or if $\out(v)<\infty$ and
$i=\lfloor \out(v)/2\rfloor$ and 
$j=\lfloor -\out(v)/2\rfloor$.
From now on we assume that all trees are rooted plane 
trees and we will simply refer to them as trees.

The set of trees is denoted by $\Tset$  and the set of trees with $n$ vertices is denoted by $\Tset_n$. A metric is defined on $\Tset$ as follows. For every $R\geq 0$ define the set
\begin {equation}
V^{[R]} = \bigcup_{n=0}^R \{\lfloor -R/2\rfloor+1,\lfloor -R/2\rfloor+2,\ldots,\lfloor R/2\rfloor-1,\lfloor R/2\rfloor\}^n
\end {equation}
and for $\Tree\in \Tset$ let $\Tree^{[R]}$ be the finite subtree of $\Tree$ with vertex set $V(\Tree)\cap V^{[R]}$. Define the metric
\begin {equation}
 d_\Tset(\Tree_1,\Tree_2) =  \left(1+\sup\left\{R~:~\Tree_1^{[R]}=\Tree_2^{[R]}\right\}\right)^{-1},\qquad \Tree_1,\Tree_2\in\Tset.
\end {equation}
The set $\Tset$ is equipped with the Borel 
$\sigma$-algebra generated by $d_\Tset$. 

To each tree $\Tree$ we associate a looptree $\Loop(\Tree)$ as
follows. The vertex set of $\Loop(\Tree)$ is the same as that of
$\Tree$ and the edges are constructed by connecting adjacent siblings
with an edge and by connecting each parent to both its leftmost and
rightmost child with an edge. Thus each vertex of finite outdegree 
$k\geq1$ in
the tree corresponds to a cycle of length $k+1$ in the looptree and a
vertex of infinite degree corresponds to a doubly--infinite path.
Note that our definition of $\Loop(\tree)$ coincides with what is
called $\Loop(\tree)$ in~\cite{curien:2013b} and 
$\Loop'(\tree)$ in~\cite{curien:2013a}.
We denote the set of looptrees by $\Lset$. 

For any rooted graph $G$ with root $\root$ and graph metric $d_G$,
the graph ball of radius $R$ centered on the root of $G$ is
defined to be
\begin {equation} \label{balldef}
 B(R;G) := \{x \in G ~:~d_G(x,\root) < R\}.
\end {equation}
We identify $B(R;G)$ with the induced subgraph of $G$
spanned by its vertices.
A metric is defined on $\Lset$ by
\begin {equation}
 d_\Lset(L_1,L_2) = 
\left(1+\sup\left\{R~:~B(R;L_1)=
B(R;L_2)\right\}\right)^{-1},\qquad L_1,L_2 \in \Lset
\end {equation}
and the set $\Lset$ is equipped with the Borel $\sigma$-algebra
generated by $d_\Lset$.  Limits of sequences in metrics of the type
$d_\Tset$ and $d_\Lset$ are often referred to as \emph{local limits}
since convergent sequences are eventually constant in every finite
`neighbourhood' of the root, where the precise definition of
neighbourhood depends on the context.  We have the following important
result:
\begin{lemma} \label{continuous} The function $\Loop: \Tset
\rightarrow \Lset$ is a homeomorphism.
\end{lemma}
\begin {proof}
 By construction, $\Loop$ is a bijection. We will outline the proof of
 the continuity of $\Loop$ and the continuity of its inverse may be
 shown in a similar way. Let $(\Tree_i)_{i\geq 0}$ be a sequence of
 trees in $\Tset$ which has a limit $\Tree$. This means that for every
 $R\geq 0$, eventually (for all $i$ large enough)
 $\Tree_i^{[R]}=\Tree^{[R]}$. Thus, for each $R\geq 0$, eventually
 $\Loop(\Tree_i^{[R]}) = \Loop(\Tree^{[R]})$. For each $R'\geq 0$ one
 may choose $R$ large enough ($R\geq 2R'+1$ suffices) 
such that eventually
 $B(R';\Loop(\Tree_i^{[R]})) = B(R';\Loop(\Tree_i))$ and such that
 $B(R';\Loop(\Tree^{[R]})) = B(R';\Loop(\Tree))$. Therefore,
 eventually
\begin {equation*}
B(R';\Loop(\Tree_i)) =  B(R';\Loop(\Tree_i^{[R]}) 
= B(R';\Loop(\Tree^{[R]})) = B(R';\Loop(\Tree)).
\end {equation*}

\end {proof}

\subsection{Simply generated trees and the local limit theorem}
\label{local_sec}

Let $(w_i)_{i\geq 0}$ be a sequence of non-negative numbers. 
Simply generated trees with $n$ vertices are random 
trees $\sgt_n$ in $\Tset$  defined by assigning a weight
\begin {equation} \label{sgtweight}
 W(\Tree) = \prod_{v\in V(\Tree)} w_{\out(v)}
\end {equation}
to each tree $\tree\in\Tset$ and letting
\begin {equation}  \label{sgtprob}
 \P(\sgt_n = \Tree) = \frac{W(\Tree)}{\sum_{\Tree'\in \Tset_n}W(\Tree')}, \quad \Tree \in \Tset_n \quad (0 ~\text{otherwise}).
\end {equation}
The random looptrees with $n$ vertices with which we are 
concerned in this paper are defined 
by $\sgloop_n = \Loop(\sgt_n)$.

Janson~\cite{janson:2012sgt} showed in full 
generality that the sequence of random trees
$\sgt_n$ converges weakly towards an infinite random tree $\sgt$ and
we state this as Theorem~\ref{th:janson} below. 
Before stating the theorem we describe the
limit $\sgt$ which requires a few definitions. Denote the generating
function of $(w_i)_{i\geq 0}$ by
\begin{equation}
g(z) = \sum_{i=0}^\infty w_i z^i
\end{equation}
 and denote its radius of convergence by $\rho$. 
Note that if $\rho > 0$ then for each $0 < t < \rho$, the sequence of probabilities
\begin {equation}
p_i(t) = \frac{t^i w_i}{g(t)}
\end {equation}
when inserted in \eqref{sgtweight} and \eqref{sgtprob}, defines the same random tree $\sgt_n$ as $w_i$. The largest mean of  $p_i(t)$ is given by 
\begin {equation} \label{defgamma}
 \gamma = \lim_{t\nearrow \rho} \frac{tg'(t)}{g(t)}.
\end {equation}
In the case $\rho =0$ we define $\gamma=0$. Let $\tau$ be
\begin {enumerate}
\item the solution $\tau\in(0,\rho)$ to $\tau g'(\tau)/g(\tau) = 1$,
  if $\gamma > 1$ (which is unique since the ratio in
  \eqref{defgamma} is increasing, see e.g.~\cite[Lemma
    3.1]{janson:2012sgt}); or
\item $\tau = \rho$, if $\gamma \leq 1$.
\end {enumerate}
Then define $\pi_i=\lim_{t\uparrow\tau} p_i(t)$
if $\rho>0$ or $\pi_i=\delta_{i,0}$ if $\rho=0$. 
(Note that the sequence $(\pi_i)_{i\geq0}$ is well-defined  
in all cases
since $g(t)$ is monotonically increasing in $t>0$.)
Let
$\off$ be a random variable distributed by $(\pi_i)_{i\geq 0}$.  One
may view $\off$ as an offspring distribution of a Galton--Watson
process and we denote its expected value by $\mean$. By definition
\begin {equation}
 \mean = \min\{\gamma,1\} \leq 1
\end {equation}
and so $\off$ is either critical or sub-critical. 
For $\rho>0$, $\sgt_n$ is a Galton--Watson tree with 
offspring distribution $\xi$ conditioned on
having $n$ vertices.
Define also the random variable $\offhat$ by
\begin {equation}\label{hat_xi_eq}
 \P(\offhat = k) = \left\{
\begin {array}{ll}
k\pi_k & \text{if}~k<\infty \\
1-\mean & \text{if}~k=\infty.
\end {array}\right.
\end {equation}
(If $\mean=1$ then $\offhat$ is simply the size-biased
version of $\off$.)

The tree $\sgt$ will now be introduced, following a construction by
Janson \cite{janson:2012sgt}. It is a modified Galton--Watson tree
which contains two types of vertices, 
called normal and special, which
independently give birth to vertices in the following recursive
way. First of all the root is declared to be special. Special vertices
give birth to vertices according to the offspring distribution
$\offhat$. If the number of children of a special vertex is finite,
one of them is chosen uniformly and declared to be special and the
rest are declared to be normal. If the number of children is infinite
(which may happen when $\mean < 1)$ then all the children are declared
to be normal. Normal vertices give birth to normal vertices according
to the offspring distribution $\off$.

\begin {theorem}[Janson \cite{janson:2012sgt}] \label{th:janson}
For any sequence $(w_i)_{i\geq 0}$ such that $w_0>0$ and $w_k >
0$ for some $k\geq 2$ it holds that 
 $\sgt_n \dto \sgt$ as  $n\to\infty$.
\end {theorem}

Special cases of Theorem \ref{th:janson} have been proven independently by many
authors. The original proof in the case $\mean = 1$ was implicitly
given by Kennedy \cite{kennedy:1975} and later Aldous and Pitman
\cite{aldous:1986} gave an explicit proof. The case $0<\mean<1$ was
originally proven, for weights obeying a power law, by Jonsson and
Stefánsson \cite{jonsson:2011} and the case $\mean = 0$ was proved by
Janson, Jonsson and Stefánsson \cite{janson:2011} in almost complete
generality.

By Theorem \ref{th:janson} and Lemma \ref{continuous} we immediately
get the corresponding convergence for the random looptrees. Define
$\sgloop = \Loop(\sgt)$.
\begin {corollary}\label{sglopp_cor}
 For any sequence $(w_i)_{i\geq 0}$ such that $w_0>0$ and $w_k >
0$ for some $k\geq 2$ it holds that 
$\sgloop_n \dto \sgloop$ as  $n\to\infty$.
\end {corollary}

The tree $\sgt$ is qualitatively different depending on whether $\mean
= 1$ or $\mean < 1$. In the former case, $\mean = 1$, the special
vertices will form an infinite path starting at the root, which we
refer to as an \emph{infinite spine}. The normal children of the
special vertices are the roots of independent critical Galton--Watson
trees with offspring distribution $\off$, and we will refer to these
trees as \emph{outgrowths} from the spine. This construction is
initially due to Kesten~\cite{kesten:1986}. In the latter case, $\mean
< 1$, the special vertices form an a.s.~finite path, which we refer to
as a \emph{finite spine} and its length $\ell$ is distributed by
\begin {equation}
\P(\ell=i) = (1-\mean)\mean^i, \quad i\geq 0
\end {equation}
(where we understand $0^0 = 1$). The finite spine starts at the root
and ends at a special vertex having infinite degree. The outgrowths
from the finite spine are defined as in the former case and are now
sub-critical Galton--Watson trees. This construction is due to Jonsson
and Stefánsson \cite{jonsson:2011}. In the special case $\mean = 0$
the spine has length 0 and the outgrowths are all empty and thus
$\sgt$ is deterministic.

From the description of $\sgt$ one may arrive at a similar description
of $\sgloop=\Loop(\sgt)$.  The spine of $\sgt$ corresponds to what may
be called a \emph{loopspine} in $\sgloop$. When $\mu = 1$ the loopspine consists of an infinite sequence of 
cycles $(C_i)_{i\geq1}$ whose lengths are independent copies of 
$\offhat+1$.  These cycles form a chain $\cC$ by recursively 
attaching $C_{i+1}$ to $C_i$ at a point $x_{i+1}$ chosen
uniformly (independently of all other choices) 
on $C_i\setminus\{x_i\}$.  The recursion starts with 
$x_1=\root$.  At each vertex of $\cC\setminus\{x_i:i\geq1\}$
we then attach an independent copy of $\Loop(\tree)$,
where $\tree$ is a Galton--Watson tree with offspring
distribution~$\xi$.  We call these copies of
$\Loop(\tree)$ \emph{outgrowths}. The case $\mu<1$ is similar however the sequence $(C_i)_{1 \leq i \leq \ell +1}$ is almost surely finite and the last element is not a cycle but a bi-infinite path (the  graph with vertex set $\mathbb{Z}$ and an edge between $i$ and $j$ iff $|i-j|=1$). In the extreme case when $\mu=0$ the looptree is deterministic and equals a rooted bi-infinite path.

In the following two subsections we provide
basic estimates on the volume growth of $\sgloop$.  These results do
not rely on any assumptions about $\xi$ being in the domain of
attraction of a stable distribution, hence we separate them from the
main proofs in Section~\ref{spec_sec}.

\subsection{A bound on the height of a random looptree}

Let $G$ be a finite graph, and single out a root vertex in $G$,
which we denote $\root$.  Letting $d_G$ denote the graph metric in
$G$, we define the \emph{height} of $G$ as
\begin{equation}
\height(G)=\max_{v\in G} d_G(\root,v).
\end{equation}
In this section we will prove the following result;
to avoid trivialities we assume that $\xi$ is not
identially 1.
\begin{lemma}\label{height_lem}
Let $\tree$ be a Galton--Watson tree with
critial offspring distribution $\xi$ 
(i.e.,  $\E(\xi)=1$).
Then there is a constant $c>0$ such that 
for all $m\geq 2$,
\begin{equation}\label{GW-loop-ht}
\P(\height(\Loop(\tree))\geq m)\leq \frac{c}{m}.
\end{equation}
\end{lemma}

Before proving this we state some general facts about (plane) trees
and their associated looptrees.
Let $\tree$ be a finite tree 
with $n$ vertices and let
$\root=u_0,u_1,u_2,\dotsc,u_{n-1}$ denote the vertices
of $\tree$ in lexicographical (depth-first-search)
order.  We define the \emph{Lukasiewicz path}
$W(\tree)=(W_j(\tree):0\leq j\leq n)$ of $\tree$ by
\begin{equation}\label{L-def}
W_0(\tree)=0,\quad 
W_{j+1}(\tree)=W_j(\tree)+\out(u_j)-1 \mbox{ for }0\leq j\leq n-1.
\end{equation}
When $\tree$ is clear from the context we simply write
$W$ and $W_j$ in place of $W(\tree)$ and $W_j(\tree)$.
We have that $W_n=-1$ and that $W_j\geq 0$
for $j\leq n-1$;  also, $W_{j+1}\geq W_j-1$ for
all $0\leq j\leq n-1$.

Note that $u_1$ is the leftmost child of the root, and that
the process $(W_j:1 \leq j < n)$ achieves a record minimum at step $j$ if
and only if $u_j$ is a child of the root.  Introducing the
number 
%\begin{equation}\label{eq:L-minima}
$M_j=\min_{1\leq i\leq j} W_i$, for $1\leq j < n$,
%\end{equation}
and  denoting the children of $\root$ by $v_1,\dotsc,v_m$,
numbered from left to right, it follows that
\begin{equation}\label{C-succ}
u_j\succeq v_k \Leftrightarrow M_j=m-k.
\end{equation}
Fix an arbitrary vertex $u_j$ in $\tree$
and for $u_i\prec u_j$ write
\begin{equation}
x(i,j)=\min_{i+1\leq \ell\leq j} W_\ell(\tree)-W_i(\tree)+1.
\end{equation}
By applying~\eqref{C-succ} iteratively to the subtrees rooted at the
vertices on the unique path from $\root$ to $u_j$ in $\tree$, we see
that 
\begin{equation}\label{loop-xij}
d_{\Loop(\tree)}(\root,u_j)\leq \sum_{u_i\prec u_j} x(i,j).
\end{equation}
(The right-hand-side of~\eqref{loop-xij} is the length of
a path in $\Loop(\tree)$
from $\root$ to $u_j$  which traverses
each `loop' in the anticlockwise direction;  
cf.\ the more general exact 
expression for $d_{\Loop(\tree)}$
in~\cite[Eq.~(40)]{curien:2013a}).

As observed in \cite[Eq.~(46)]{curien:2013a}
we have for any fixed $j\geq 1$ that
\begin{equation}\label{xij_sum}
\sum_{u_i\prec u_j} x(i,j)= d_\tree(\root,u_j) + W_j(\tree).
\end{equation}
Combining~\eqref{xij_sum} with~\eqref{loop-xij} it follows that 
\begin{equation}\label{eq:height}
\height(\Loop(\tree))\leq 
\height(\tree) + \max_{0\leq j\leq n} W_j(\tree).
\end{equation}

\begin{proof}[Proof of Lemma~\ref{height_lem}]
By~\eqref{eq:height}
\begin{equation}
\P(\height(\Loop(\tree))\geq m)\leq \P(\height(\tree)\geq m/2)
+\P(\max W(\tree)\geq m/2).
\end{equation}
Writing $\xi_i=\out(u_i)$ for the summands appearing
in~\eqref{L-def}, the $\xi_i$ are simply independent copies 
of $\xi$ and 
$W_j=\sum_{i=1}^j (\xi_i-1)$ is a random walk.
The event 
$\{\max W(\tree)\geq m/2\}$ equals the event that the random walk
$W_j$ reaches level $\geq m/2$ before it reaches level
$-1$.  Letting 
\[
T=\min\{j\geq0: W_j \leq -1 \mbox{ or } W_j\geq m/2\},
\]
it follows from Wald's identity (see e.g.~\cite[page 601]{feller:vol2}) that
$\E(W_T)=\E(T)\E(\xi-1)=0$.
Since $W_T+1\geq0$ we have by Markov's inequality
that
\begin{equation}\label{eq:kol2}
\P(\max W(\tree)\geq m/2)=
\P(W_T\geq m/2)\leq
\frac{\E(W_T)+1}{m/2+1}=\frac{1}{m/2+1}.
\end{equation}
Moreover, it follows 
from~\cite[Theorem~I.9.1]{athreya-ney}
 (and the remark below it)
that
\[
\P(\height(\tree)\geq m/2)\leq \frac{c'}{m},
\]
for a constant $c'>0$.  The result follows.
\end{proof}

\subsection{A bound on the volume of a random loopspine} \label{ss:loopspine}

As in the description of $\sgloop$ in Section~\ref{local_sec},
let $(C_i)_{i\geq 1}$ be a sequence of cycles, each having a
marked vertex $x_i$.  Assume that 
$C_i$ has a random length $X_i+1$, where the $X_i$ are
independent and  distributed
as a random variable $X$ taking values in the positive integers
$\{1,2,\ldots\}$.  Construct an infinite chain $\mathcal{C}$
by identifying the point $x_{i+1}$ in cycle $C_{i+1}$ to a uniformly
chosen point in $C_i\setminus\{x_i\}$ for all $i\geq 1$.
We call $x_i$ the root of the cycle $C_i$, and let 
$\root=x_1$ be the root of $\cC$.
The loopspine of $\sgloop$ is constructed like this with 
$X=\offhat$, but the following bound holds
regardless of the distribution of $X$
(as long as $X\geq1$ a.s.).  Define the set
\begin {equation} \label{An_def}
A_n = \{v \in \mathcal{C} ~:~ d_\mathcal{C}(\root,v)\leq n\}
\end {equation} 
where $d_\mathcal{C}$ denotes the graph metric in
$\mathcal{C}$. 

\begin {lemma} \label{l:loopspinevol}
For all $n\geq 0$
\begin {equation}
 \E(|A_n|) \leq 16 n +1.
\end {equation}
\end {lemma}
\begin {proof}
Given $(X_i)_{i\geq 1}$, let $(U_i(X_i))_{i\geq1}$ denote independent
random variables, $U_i(X_i)$ being uniformly chosen from 
the set $\{1,2,\ldots,X_i\}$.  Define
\begin {equation}\label{Yn_def}
 Y_i^{(n)} = \min\{U_i(X_i), X_i-U_i(X_i)+1,n\}.
\end {equation}
Letting
\begin {equation}
 T_n = \min\Big\{N ~:~ \sum_{i=1}^N Y_i^{(n)} \geq n\Big\},
\end {equation}
we see that $T_n$ is the smallest number $N$ for which the root of a 
cycle $C_{N+1}$ is at a distance at least $n$ from the root $\root$. 
Writing $X_i^{(n)} = \min\{X_i,2n\}$ we have that
\begin {equation}
 |A_n| \leq \sum_{i=1}^{T_n} X_i^{(n)}+1.
\end {equation}
Using this along with Wald's identity we get
\begin {equation}
 \E(|A_n|) \leq \E(T_n)\E(X_1^{(n)})+1.
\end {equation}
Again, by Wald's identity and the definitions of $T_n$ and $Y_i^{(n)}$ we find that
\begin {equation}
 \E(T_n) = \frac{\E\left(\sum_{i=1}^{T_n}  Y_i^{(n)}\right)}{\E(Y_1^{(n)})} = 
\frac{\E\left(\sum_{i=1}^{T_n-1} Y_i^{(n)}+  Y_{T_n}^{(n)}\right)}{\E(Y_1^{(n)})} 
\leq \frac{2n}{\E(Y_1^{(n)})}
\end {equation}
and thus
\begin {equation}\label{exp_ratio}
  \E(|A_n|)  \leq  2n \frac{\E(X_1^{(n)})}{\E(Y_1^{(n)})}+1.
\end {equation}
 We conclude by showing that the ratio 
of expected values in~\eqref{exp_ratio} is bounded from the above by 8. 
Fix $\epsilon \in(0, 1/2)$ and note that
\begin {eqnarray*}
 Y_1^{(n)} &\geq& Y_1^{(n)} \indic{\epsilon X_1 \leq U_1(X_1) \leq (1-\epsilon) X_1+1} \\
  &\geq& \epsilon  \min\{X_1, n/\epsilon\} \indic{\epsilon X_1 \leq U_1(X_1) \leq (1-\epsilon) X_1+1}
\end {eqnarray*}
and thus
\begin {eqnarray} \nonumber
 \E(Y_1^{(n)}) &\geq& \epsilon \E(\min\{X_1, n/\epsilon\} \indic{\epsilon X_1 \leq U_1(X_1) \leq (1-\epsilon) X_1+1}) \\ \nonumber
 &=& \epsilon \E\big(\min\{X_1, n/\epsilon\} \P(\epsilon X_1 \leq U_1(X_1) \leq (1-\epsilon) X_1+1~|~X_1)\big) \\ \nonumber
 &=& \epsilon \E\Big(\min\{X_1,n/\epsilon\} \frac{\lfloor (1-\epsilon)X_1\rfloor  - \lceil \epsilon X_1\rceil +2}{X_1}\Big) \\
 &\geq& \epsilon(1-2\epsilon)\E(\min\{X_1,n/\epsilon\}).
\end {eqnarray}
Also, $\E(X_1^{(n)}) \leq \E(\min\{X_1,n/\epsilon\})$
since $\epsilon < 1/2$.  Taking the optimal $\epsilon=1/4$ yields the desired result.
\end {proof}
We conclude this section by establishing a lower bound on the number
of outgrowths (possibly empty) from the loopspine $\mathcal{C}$ up to
and including distance $n$ from the root. We use the same notation as
was introduced in the beginning of this section. We will call the
vertices $\root=x_1,x_2,\ldots,$ \emph{closed} and the vertices in
$\mathcal{C}\setminus\{x_1,x_2,\ldots\}$ \emph{open}. By definition,
outgrowths only emanate from the open vertices and thus we want a
lower bound on the number of open vertices up to and including
distance $n$ from the root. Denote their number by $\open_n$. For each
$i\geq 0$, define $Y_i = \min\{U_i(X_i),X_i-U_i(X_i)+1\}$ with the
same $U_i(X_i)$ as in \eqref{Yn_def}.
\begin {lemma} \label{l:openbound}
Let $p = \P(Y_i > 1)$. Then $\open_n$ stochastically dominates a
binomial random variable $\mathrm{Bin}(\lfloor n/2\rfloor,p)$.
\end {lemma}
\begin {proof}
Let $K$ be the index of the first of the cycles $C_1,C_2,\dotsc$
to reach distance $n$ from the root.  Then we have that 
$\mathcal{R}_n\geq n-K$.
For all $k\geq1$ we have that
 $\{K\leq k\}\supseteq\{Y_1+\dotsb+Y_k\geq n\}$,
and thus 
\begin{equation}\begin{split}
\P(K\leq k) &\geq \P(Y_1+\dotsb+Y_k\geq n)\\
&= \P((Y_1-1)+\dotsb+(Y_k-1)\geq n-k)\\
&\geq \P(\eps_1+\dotsb+\eps_k\geq n-k),
\end{split}\end{equation}
where $\eps_i=\indic{Y_i>1}$.  Thus for 
each $0\leq r\leq \lfloor n/2\rfloor$ we have that
\begin{equation}\begin{split}
\P(\mathcal{R}_n\geq r)&\geq\P(K\leq n-r)\\
&\geq 
\P(\eps_1+\dotsb+\eps_{n-r}\geq r)\\
&\geq 
\P(\eps_1+\dotsb+\eps_{\lfloor n/2\rfloor}\geq r).
\end{split}\end{equation}
This gives the result since $\eps_1+\dotsb+\eps_{\lfloor n/2\rfloor}$
has the law Bin($\lfloor n/2\rfloor,p$).
\end {proof}

\section{Random walk and spectral dimension}
\label{spec_sec}

In this section we prove our main results on random walk on
the infinite random looptree $\sgloop$. 
Given a realization of $\sgloop$, let $(X_n:n\geq0)$ 
denote simple random
walk on $\sgloop$ started at $\root$.  Apart from the
return probability 
$p_{2n}^\sgloop(\root,\root)=\P(X_{2n}=\root\mid\sgloop)$ 
we will also consider the escape time from a ball of radius $R$ defined by
$\tau_R(\sgloop):=\min\{n\geq0:X_n\not\in B(R;\sgloop)\}$
and the (quenched) expected escape time
$T_R(\sgloop)=\E(\tau_R\mid\sgloop)$.

We will focus on the case when the
offspring distribution $\off$ (defined above~\eqref{hat_xi_eq}) is
critical (i.e.\! $\mu = \E(\off)=1$) and is in the domain of attraction
of a stable law with index $\alpha\in(1,2]$
(i.e.\! satisfies~\eqref{stableintro}). We note briefly that in
the sub-critical case $\mu<1$ the looptree is a bi-infinite path
with small outgrowths and a random walker perceives it as
one-dimensional:  one may easily show that when $\mu=0$ or $0<\mu<1$
and $\E(\xi^{1+\epsilon})<\infty$ for some $\epsilon>0$ then almost
surely $d_\mathrm{s}(\sgloop)=1$. From now on we thus consider the
critical case only. 

Denote by $|\tree|$ the total number
of individuals in
the Galton--Watson tree $\tree$.  We begin by summarizing
some facts about the tail probabilities and
generating functions of $\off$ and $|\tree|$
which follow from the assumption that $\off$ 
is in the domain of attraction of a stable 
law with index $\alpha \in (1,2]$.
 Recall that a function $L$ is said to vary slowly at infinity if it is measurable and for any $\lambda \in \mathbb{R}$ it holds that
\begin {equation}
 \lim_{x\rightarrow\infty} \frac{L(\lambda x)}{L(x)} = 1.
\end {equation}
Common examples of slowly varying functions are powers of logarithms and iterates of logarithms.
By~\cite[Theorem~XVII.5.2]{feller:vol2} 
one may write
\begin {equation}\label{truncvar}
 \E(\off^2\indic{\off\leq x}) =  x^{2-\alpha} L_1(x)
\end {equation}
where $L_1$ is slowly varying at infinity,
cf.~\eqref{stableintro}.
When $\alpha<2$ we then have 
\begin {equation}\label{stabletail}
\P(\off > x) \sim \frac{2-\alpha}{\alpha}x^{-\alpha}L_1(x).
\end {equation}
(See~\cite[Corollary~XVII.5.2 and~(5.16)]{feller:vol2}.)
We will denote the generating function 
of the offspring probabilities by
\begin {equation}\label{genf}
 f(s) = \E(s^\off) = \sum_{n=0}^\infty \pi_n s^n.
 \end {equation}
It satisfies
\begin {equation} \label{frel}
 f(s) = s+(1-s)^\alpha L((1-s)^{-1})\qquad s\in[0,1)
\end {equation}
where 
\begin {equation}\label{q1}
 {L}(x) \sim \frac{\Gamma(3-\alpha)}{\alpha(\alpha-1)} 
L_1(x)-\frac{1}{2}x^{\alpha-2}\mbox{ as } x\rightarrow \infty.
\end {equation}
This is standard but we sketch a proof in 
Lemma~\ref{taubalpha2} in the Appendix.
In the case $\alpha=2$ one sees 
immediately from \eqref{truncvar} that either $\off$ has finite variance and
\begin {equation}\label{q2}
 L_1(x) \rightarrow 1+f''(1) 
\qquad \text{as}~ x\rightarrow \infty.
\end {equation}
or $L_1(x)$ diverges as $x\rightarrow \infty$.

Next,  we have by Lemma \ref{l:sizeGW} 
of the Appendix that 
\begin {equation}\label{prog}
  \E(s^{|\tau|}) = 1-(1-s)^{1/\alpha}L^\ast((1-s)^{-1}) \qquad s\in[0,1)
 \end {equation}
where $L^\ast$ is slowly varying at infinity and satisfies
\begin{equation}\label{Last}
\lim_{n\rightarrow\infty} L^\ast(n)^\alpha L(n^{\frac{1}{\alpha}}L^\ast(n)^{-1}) = 1.
\end{equation}
%taubalpha1 and Remark 6.11
Let $a_n$ be a sequence such that
\begin {equation} \label{scaling}
a_n^{-1/\alpha} n L^\ast(a_n)\rightarrow 1 
\quad \text{as}~n\rightarrow \infty.
 \end {equation}
Note that \eqref{Last} and \cite[Theorem 1.5]{seneta} imply that 
\begin {equation} \label{anslowly}
 a_n \sim n^{\alpha} L(n)^{-1}
\end {equation}
and by \cite[$5^\circ$, Section 1.5]{seneta} we may and will choose $a_n$ to be strictly increasing.

\subsection{Proofs of the main results}

To prove our main results Theorems~\ref{specdim_thm}
and~\ref{specdim_thm_annealed}
we will estimate the volume- and resistance
growth in $\sgloop$, and apply the recent 
results of Kumagai and Misumi~\cite{kumagai:2008}.
Recall that if $G=(V,E)$ is a locally finite graph with 
vertex set $V$ and edge set $E$, and $A,B\subseteq V$
are disjoint subsets of $V$, then the 
\emph{effective resistance} 
$\Reff(A,B)$ between $A$ and $B$ is defined by
letting $\Reff(A,B)^{-1}$ be the infimum  of
\begin{equation}
\sum_{xy\in E}(h(x)-h(y))^2
\end{equation}
over all functions $h:V\to\RR$ satisfying 
$h(a)=1$ for all $a\in A$ and
$h(b)=0$ for all $b\in B$.  Here 
we will mainly be using  the following two very
simple facts about effective resistances:
\begin{enumerate}
\item for all $x,y\in V$ we have 
$\Reff(x,y)\leq d_G(x,y)$, and
\item if $C\subseteq V$ \emph{separates} $A$ from $B$,
in the sense that every path in $G$ between some 
$a\in A$ and $b\in B$ contains some $c\in C$, then
$\Reff(A,B)\geq \Reff(A,C)$.
\end{enumerate}
We will also use the standard
series and parallel laws,
for which we refer to~\cite[Chapter~2]{lyons-peres}.
Recall the graph ball of radius $n$ defined in \eqref{balldef}.
We define its \emph{volume} $V_n=|B(n; \sgloop)|$
as the number of vertices in it.  

Let $v(n) = a_n$ and $r(n) = n$ and let $\mathcal{I}$ be the inverse
function of $v\cdot r$. 
The functions $v$ and $r$ are both strictly increasing and in light of
\eqref{anslowly} and Lemma \ref{l:powerratio} there exist $1 \leq d_1
< d_2$, $0<\alpha_1,\alpha_2 \leq 1$ and $C_1,C_2\geq 1$ such that
\begin {equation}
 C_1^{-1}\left(\frac{n}{n'}\right)^{d_1}\leq \frac{v(n)}{v(n')}\leq
 C_1\left(\frac{n}{n'}\right)^{d_2}
 ~\text{,}~C_2^{-1}\left(\frac{n}{n'}\right)^{\alpha_1}\leq
 \frac{r(n)}{r(n')}\leq C_2\left(\frac{n}{n'}\right)^{\alpha_2}
\end {equation}
for all $0< n' \leq n <\infty$. Thus, $v$ and $r$ satisfy the basic
conditions required in \cite{kumagai:2008}.
(In~\cite{kumagai:2008} the `volume' is defined slightly differently
than our $V_n$, but since all vertices in the graph $\sgloop$ have
uniformly bounded degrees our definition may be used equivalently.)
We will prove that there exist $c_1,c_2,c_3>0$, $\lambda_0>0$
and $q_1>0$, $q_2>2$ such that for all $\lambda\geq\lambda_0$
and all $n\geq1$ we have
\begin{align} 
&\P(V_n\leq \l v(n), \Reff(\root,B(n;\sgloop)^c)\geq\l^{-1}r(n))
  \geq 1-\frac{c_1}{\l^{q_1}} \label{J1}\\
&\P(V_n\geq \l^{-1}v(n), \forall y \in
B(n;\sgloop), \Reff(\root,y)\leq \l r(d_\sgloop(\root,y))) 
 \geq 1-\frac{c_2}{\l^{q_2}} \label{J2}\\
& \E(\Reff(\root,B(n;\sgloop)^c)V_n)\leq c_3 v(n) r(n) \label{J3}.
\end{align}
%Following~\cite{kumagai:2008} we define for
%each $\l>1$ the random set
%\begin{equation}\label{J_def}
%\begin{split}
%J(\l)=\{n\geq 0:\;& \l^{-1}v(n)\leq V_n\leq \l v(n),\\ &
%\Reff(\root,B(n;\sgloop)^c)\geq\l^{-1}r(n),\\ &\exists y \in
%B(n;\sgloop): \Reff(\root,y)\leq \l r(d_\sgloop(\root,y))\}.
%\end{split}
%\end{equation}
Applying Theorem~1.5 and Proposition~1.4 of~\cite{kumagai:2008}
we get the following from~\eqref{J1}--\eqref{J3}:
\begin{theorem}\label{KM_thm}
There exist $\beta_1,\beta_2>0$ such that the following hold.
\begin{enumerate}
\item For each realization of $\sgloop$ there exists a number
$N(\sgloop)<\infty$ such that
\[
\frac{(\log n)^{-\beta_1}}{v(\mathcal{I}(n))}\leq
p_{2n}^\sgloop(\root,\root)\leq
\frac{(\log n)^{\beta_1}}{v(\mathcal{I}(n))},\quad
\mbox{for } n\geq N(\sgloop).
\]
\item For each realization of $\sgloop$ there exists a number
$R(\sgloop)$ such that
\[
(\log R)^{-\beta_2}v(R)r(R)\leq T_R(\sgloop) \leq
(\log R)^{\beta_2}v(R)r(R),\quad
\mbox{for } R\geq R(\sgloop).
\]
\item There exist $C_1,C_1'>0$ such that for all $n\geq 1$
\[
\frac{C_1}{v(\mathcal{I}(n))}\leq
\E(p_{2n}^\sgloop(\root,\root))\leq
\frac{C_1'}{v(\mathcal{I}(n))}.
\]
\item There exist $C_2,C_2'>0$ such that for all $R\geq 1$
\[
C_2v(R)r(R)\leq \E(T_R(\sgloop)) \leq C_2'v(R)r(R).
\]
\end{enumerate}
\end{theorem}

Our results Theorem~\ref{specdim_thm}
and Theorem~\ref{specdim_thm_annealed}
on the spectral dimension
follow immediately from the first and third parts of Theorem~\ref{KM_thm},
respectively, using~\eqref{anslowly} and standard results on slowly varying functions given in the Appendix.
In fact,~\cite{kumagai:2008} shows that some
more results follow from the 
inequalities~\eqref{J1}--\eqref{J3} which we
do not state explicitly here.
The rest of this section will be devoted to
proving~\eqref{J1}--\eqref{J3}.

\subsection{Bounds on the volume}
In this section we provide a number of estimates on 
the volume of the random looptrees under consideration. 
Recall that $a_n=v(n)$.
\begin {lemma}\label{exp_inv_vol}
For any $\gamma > 0$ there is a constant $K_\gamma$ such that
\begin{equation}\label{vol_inv}
\E[V_n^{-\gamma}]\leq K_\gamma a_n^{-\gamma}.
\end{equation}
\end {lemma}

\begin {proof}
We will start by
establishing the following:
for every $\delta \in(0,1/\alpha)$ 
there are constants $c_1,c_2,c_3,c_4>0$,
possibly depending on $\delta$, such that
\begin{equation}\label{stretched}
\P(V_n<\lambda^{-1} a_n)\leq
c_1 \exp(-c_2 \lambda^{1/\alpha-\delta})
\quad\mbox{ whenever }
c_3\leq \lambda\leq c_4 a_n.
\end{equation}
We let $X^{(n)}$ denote the number of vertices
of $\Loop(\tree)$ at distance at most $n$
from the root of $\Loop(\tree)$, and we let
$X^{(n)}_1,X^{(n)}_2,\dotsc$ denote independent
copies of $X^{(n)}$.  Note that $V_n$
dominates a sum $\sum_{i=1}^{\open_{\lfloor n/2\rfloor}}X^{(\lfloor n/2\rfloor)}_i$ with $\open_n$ from Lemma \ref{l:openbound}.
Thus it suffices to show that there are 
constants $c'_1,c'_2,c'_3,c'_4>0$ such that
\begin{equation}\label{stretched2}
\P\Big(\sum_{i=1}^{\open_n} X^{(n)}_i<
\lambda^{-1} a_n\Big)\leq
c_1'\exp(-c_2'\lambda^{1/\alpha-\delta})
\end{equation}
whenever $c_3'\leq \lambda\leq c_4'a_n$. Note that by Lemma \ref{l:openbound} it holds that for $x \in [0,1]$
\begin {equation} \label{eq:open}
 \E(x^{\open_n}) \leq (1-p(1-x))^{\lfloor n/2\rfloor}
\end {equation}
and $p>0$. Let $t=\lambda a_n^{-1}$. Using Markov's inequality, the independence of the $X_i^{(n)}$'s and $\open_n$, Eq.~\eqref{eq:open}, Lemma \ref{height_lem} and Eq.~\eqref{prog}
we have for any $t>0$ that
\begin{equation}
\begin{split}
\P\Big(\sum_{i=1}^{\open_n} X_i^{(n)}<\lambda^{-1} a_n\Big)&=
\P\Big(\exp\Big(-t\sum_{i=1}^{\open_n} X_i^{(n)}\Big)>
e^{-1}\Big)\\
&\leq e
\E\Big(\E\big[\exp\big(-t X^{(n)}\big)\big]^{\open_n}\Big)\\
&\leq e
\Big(1-p\big(1-\E\big[\exp\big(-t X^{(n)}\big)\big]\big)\Big)^{\lfloor n/2\rfloor}\\
&\leq e
\Big(1-p\big(1-\big[\E(e^{-t|\tree|})+\P(\height(\Loop(\tree))>n)\big]\big)\Big)^{\lfloor n/2\rfloor}\\
&\leq e
\Big(1-p\big((1-e^{-t})^{1/\alpha} L^\ast((1-e^{-t})^{-1})-c/n\big)\Big)^{\lfloor n/2\rfloor}.
\end{split}
\end{equation}
Recall that 
$\lambda\leq c_4'a_n$ for some constant $c_4'>0$.
By Lemma~\ref{l:powerslow}
there is a slowly varying function $\tilde L$
asymptotically equal to $L^\ast$ such that
$x^{-1/\alpha}\tilde L(x)$ is non-increasing.
Hence, for any $\delta$ sufficiently small one may choose $c_4'$ and hence $t$ small enough such that
\begin {equation}
\begin{split}
(1-e^{-t})^{1/\alpha}&L^\ast((1-e^{-t})^{-1}) \geq
\tfrac{1}{2}(1-e^{-t})^{1/\alpha}\tilde L((1-e^{-t})^{-1}) \\
&\geq k_1 t^{1/\alpha}\tilde L(2/t) 
= k_1 \frac{\lambda^{1/\alpha}}{n} 
\frac{n \tilde L(a_n) }{a_n^{1/\alpha}} 
\frac{\tilde L(\frac{2}{t}) }{\tilde L(\frac{\lambda}{t})} \\
&\geq k_2 \frac{\lambda^{1/\alpha-\delta}}{n}
\end {split}
\end{equation}
for some constants $k_1,k_2>0$. In the second step we used that $1-e^{-t} > t/2$ for $t$ small
enough, and in the last step we used~\eqref{scaling}
and Lemma~\ref{l:powerratio}.

Finally, for $\lambda\geq c_3'$ with $c_3'$ 
large enough that $k_2\lambda^{1/\alpha-\delta}\geq c$
(the constant from Lemma~\ref{height_lem})
we have
\begin{equation}
\P\Big(\sum_{i=1}^{\open_n} X_i^{(n)}<\lambda^{-1} a_n\Big)\leq
e(1-p(k_2\lambda^{1/\alpha-\delta}-c)/n)^{\lfloor n/2 \rfloor}\leq
e\exp\Big(-\frac{p}{3}\big(k_2\lambda^{1/\alpha-\delta}-c\big)\Big),
\end{equation}
proving~\eqref{stretched2} and hence~\eqref{stretched}.

We now show that~\eqref{stretched} 
implies~\eqref{vol_inv}, using an argument similar
to one in~\cite{fuji:2008}.
Let $m$ be a fixed integer large enough such
that $c_3 \leq c_4 a_m$ and let $\gamma > 0$.  
Then by~\eqref{stretched}
\begin{equation}
\begin{split}
\E[V_n^{-\gamma}]&= \E\Big[V_n^{-\gamma}\indic{V_n^{-1}\leq c_4 a_m/a_n}\Big] +\sum_{k=m}^{n-1} \E\Big[V_n^{-\gamma}\indic{c_4 a_k
    /a_n<V_n^{-1}\leq c_4a_{k+1}/a_n}\Big]\\ &\quad+
\E\Big[V_n^{-\gamma}\indic{V_n^{-1}> c_4a_n/a_n}\Big]\\ 
&\leq (c_4 a_m)^\gamma a_n^{-\gamma}+ 
c_1 c_4^\gamma a_n^{-\gamma}\sum_{k=m}^{n-1}a_{k+1}^\gamma 
\exp(-c_2 (c_4 a_k)^{1/\alpha-\delta}) \\
&\quad+ c_1\exp(-c_2 (c_4 a_n)^{1/\alpha-\delta}) \leq K_\gamma a_n^{-\gamma},
\end{split}
\end{equation}
as required.  Here we used that $a_n$ is increasing, and that
from~\eqref{anslowly} and Lemma \ref{l:powerslow} 
it follows that for any $\epsilon > 0$ there are constants
$C_1,C_2>0$ (possibly depending on $\epsilon$) such that
\begin {equation}
 C_1 n^{\alpha-\epsilon} < a_n < C_2 n^{\alpha+\epsilon}
\end {equation}
so $\exp(-c_2 (c_4 a_n)^{1/\alpha-\delta})$ decays faster than any power
of $n$.   The sum
in the second last line thus converges as $n\rightarrow \infty$ and
the term following the sum is negligible.
\end{proof}

Applying Markov's inequality and using the preceding Lemma one finds that for every $\gamma > 0$ there is a constant $c_\gamma >0$ 
such that for all $n\geq 0$ and all
$\l>1$ we have
\begin{equation} \label{lvolresult}
\P(V_n<\lambda^{-1}a_n)\leq c_\gamma \lambda^{-\gamma}.
\end{equation}
This establishes the part of~\eqref{J2} regarding the volume.

As in the preceeding proof, let $X^{(n)}$ denote the number of vertices of $\Loop(\tree)$ at distance at most $n$ from the root. Recall that $f(s) = \E(s^\off)$. Then the following holds.

\begin{lemma}\label{exp_vol}
As $n\rightarrow \infty$, 
 \begin {equation}
  \E(X^{(n)}) \sim M n^{-1} a_n 
 \end {equation}
where
\begin {equation}
M = \left\{\begin{array}{c l}
     \frac{2f''(1)}{3+h(1)+f''(1)}& ~\text{if}~f''(1)<\infty\\
      \frac{2^{\alpha-1}}{\Gamma(\alpha)} & ~\text{otherwise,}
    \end{array}
\right.
\end {equation}
and
\begin {equation} \label{oddg}
 h(x) = \sum_{i=0}^\infty \pi_{2i+1}x^i.
\end {equation}

\end{lemma}
\begin {proof}
One has $\E(X^{(0)}) = 1$ and using the independence structure of $\Loop(\tree)$ one immediately arrives at the recursion
 \begin {align} \nonumber
  \E(X^{(n)}) &= \sum_{i=0}^\infty \pi_i \Big(1+2\sum_{j=1}^{\lfloor \frac{i}{2}\rfloor\wedge n} \E(X^{(n-j)})+\indic{i ~\text{is odd}}\indic{\lfloor \frac{i}{2}\rfloor+1\leq n}\E(X^{(n-\lfloor \frac{i}{2}\rfloor-1)})\Big) \\ \nonumber
  &= 1 + 2\sum_{i=0}^{2n-1} \pi_i \sum_{j=1}^{\lfloor \frac{i}{2}\rfloor} \E(X^{(n-j)})+2\sum_{i=2n}^{\infty} \pi_i \sum_{j=1}^{n} \E(X^{(n-j)}) \\
  &+\sum_{m=0}^{n-1}\pi_{2m+1}\E(X^{(n-m-1)}) = 1 +\sum_{j=1}^{n}  \Big(2\sum_{i=2j}^\infty \pi_i+\pi_{2j-1}\Big)\E(X^{(n-j)})  \label{volrec}
 \end {align}
for $n\geq 1$, where in the last step we swapped the $j$ and $i$ sums and renamed $m=j-1$ in the last sum. 

Next, define the generating function
\begin {equation}
 F(x) = \sum_{n=0}^\infty \E(X^{(n)}) x^n.
\end {equation}
Multiplying both sides of \eqref{volrec} with $x^n$, 
summing over $n\geq 0$ and swapping the sum over $n$ 
and $j$ yields a simple equation for $F(x)$ which has solution
\begin {equation} \label{fsolution}
 F(x) = \frac{1}{(1-x)(1-G(x))}.
\end {equation}
with
\begin {equation}
 G(x) = \sum_{j=1}^\infty \Big(2\sum_{i=2j}^\infty \pi_i + \pi_{2j-1}\Big)x^j.
\end {equation}
Swapping the $j$ and $i$ sum and some rewriting gives
\begin {equation}
 G(x) = \frac{2x}{1-x}\Big(1-f(x^{1/2}) - \frac{(1-x^{1/2})^2}{2} h(x)\Big)
\end {equation}
with $f$ from \eqref{genf} and $h$ from \eqref{oddg}. 
Inserting this expression into \eqref{fsolution} gives
\begin{equation}
F(x)^{-1}=1-x-2x\Big((1-x^{1/2})+(x^{1/2}-f(x^{1/2}))
-\tfrac{1}{2}(1-x^{1/2})^2h(x)\Big).
\end{equation}
Expanding the first term $1-x^{1/2}$ to second order and
applying~\eqref{frel} gives
\begin{equation}
F(x)^{-1}\sim (1-x)^2(1-\tfrac{x}{4})+
2(1-x^{1/2})^\alpha L\big((1-x^{1/2})^{-1}\big)+
(1-x^{1/2})^2h(x),
\end{equation}
as $x\rightarrow 1^-$.
By expanding the terms $1-x^{1/2}$ to first order this yields
\begin {equation}
 F(x) \sim (1-x)^{-\alpha} \Big((3/4+h(1)/4)(1-x)^{2-\alpha}+2^{1-\alpha}L((1-x)^{-1})\Big)^{-1}
\end {equation}
as $x\rightarrow 1^-$. The expression in the large parentheses is
clearly slowly varying as $x\rightarrow 1^-$,
and $L((1-x)^{-1})\to\tfrac{1}{2}f''(1)$ when 
$f''(1)<\infty$, by~\eqref{q1}--\eqref{q2}. 
The result therefore follows by
applying the Tauberian Theorem~\ref{thm:tauberian} along with 
\eqref{anslowly} and the fact
that $\E(X^{(n)})$ is increasing in $n$.
\end {proof}
From the preceding lemma we get the following bounds 
on $\E(V_n)$ which prove~\eqref{hausdim}.
\begin {lemma} \label{l:hausdorff}
 There are constants $k_1$ and $k_2$ such that
 \begin {equation} \label{expectedvol}
  k_1 a_n \leq \E(V_n) \leq k_2 a_n.
 \end {equation}
 for all $n \geq 0$.
\begin {proof}
Let $\mathcal{C}$ be the loopspine of
$\sgloop$  and let $A_n$ be defined as in~\eqref{An_def}. 
Let $(v_i)_i$ be a list of the vertices in $A_n\setminus\{x_j:j\geq 1\}$ and denote the finite outgrowth from $v_i$
by $\Loop(\tree_i)$. Let $X_{i}^{(n)}$ be the number of vertices in $\Loop(\tree_i)$ at distance at most $n$ from the root of $\Loop(\tree_i)$.  Recall that
$(\tree_i)_{i}$ is a sequence of independent Galton--Watson trees with offspring
distribution $\off$, which are furthermore independent of $|A_n|$.
For the upper bound we note that
\begin {equation}
 V_n \leq \sum_{i=1}^{|A_n|} X_{i}^{(n)} + |A_n|
\end {equation}
and thus by Wald's lemma, Lemma \ref{l:loopspinevol} and Lemma \ref{exp_vol}
\begin {equation}
 \E(V_n) \leq \E(|A_n|)\E(X_{i}^{(n)}+1) \leq k_2 a_n.
 \end{equation}
 For the lower bound we note that $V_n$
dominates a sum $\sum_{i=1}^{\open_{\lfloor n/2\rfloor}}X^{(\lfloor n/2\rfloor)}_i$ 
with $\open_{n}$ from Lemma \ref{l:openbound} obeying $\E(\open_{n}) \geq p \lfloor n/2\rfloor$ and $p>0$. Since $\open_{\lfloor n/2\rfloor}$ is independent of the $X^{(\lfloor n/2\rfloor)}_i$'s the result thus follows from Wald's lemma and Lemma \ref{exp_vol}.
\end {proof}
\end {lemma}

By applying Markov's inequality and using the previous Lemma one finds that there is a constant $c>0$ such that for all $n\geq 1$ and all $\lambda > 1$
\begin {equation}
 \P(V_n > \lambda a_n) \leq c \lambda^{-1}.
\end {equation}
This establishes the part of~\eqref{J1} regarding the volume. Note that, since $\Reff(\root,B(n;\sgloop)^c)\leq n=r(n)$
always, the lemma also implies inequality~\eqref{J3}.

\subsection{Bounds on the resistance}
Having dealt with the volume bounds in~\eqref{J1}
and~\eqref{J2} we
now turn to the bounds on the resistance.
First we note
that the upper bound is trivial: 
since $\Reff(\root,v) \leq d_\sgloop(\root,v)$ 
for all $v \in V(\sgloop)$ we have for  
all $n\geq1$ and $\lambda>1$ that
\begin {equation}
 \P(\exists v \in B(n;\sgloop)~:~ \Reff(\root,v) > \lambda
 d_\sgloop(\root,v) ) = 0,
\end {equation}
which together with
\eqref{lvolresult} proves~\eqref{J2}.
Therefore~\eqref{J1}, and hence 
Theorem~\ref{KM_thm}, follows once we prove the following
lower bound on the resistance:
\begin{lemma}\label{l:lres}
For any $q\in(0,\alpha-1)$ there is a constant
$c>0$ such that
\begin{equation}
\P(\Reff(\root,B(n;\sgloop)^c)< \lambda^{-1}n) 
\leq c\lambda^{-q}.
\end{equation}
\end{lemma}
\begin{proof}
Let $(\offhat_i)_{i \geq 1}$ be a
sequence of independent copies of 
$\offhat$, and as before let $(C_i)_{i\geq 1}$ be a sequence of
cycles with lengths $(\offhat_i+1)_{i\geq 1}$ with marked
vertices $x_i$, joined together to
form an infinite chain $\cC$.  Recall that $\sgloop$
is formed by attaching outgrowths at the
vertices $\cC\setminus\{x_i:i\geq1\}$, these
being independent copies of $\Loop(\tree)$
where $\tree$ is a Galton--Watson tree with
offspring distribution $\off$.
 
In order to bound the resistance
$\Reff(\root,B(n;\sgloop)^c)$ we aim to find a set of
vertices in $\mathcal{C}$ of size at most 2 which (i) separates $\root$ from 
all vertices of $\sgloop$ at distance at 
least $n$ from $\root$, and (ii)
is on average sufficiently far away from $\root$ (see \eqref{dn_prob_bound} for a precise statement).  
Note that if $v\in\cC$ then 
$d_\sgloop(\root,v)=d_\cC(\root,v)$.
We will consider only the cycles $C_i$ up to
and including the first one which contains a vertex at
distance $\geq n/2$ from $\root$.
Writing 
$Y_i = \min\{U_i(\offhat_i), \offhat_i-U_i(\offhat_i)+1\}$,
with the $U_i(\offhat_i)$ uniform in $\{1,2,...,\offhat_i\}$
as in~\eqref{Yn_def}, the first $C_i$ intersecting
level $\geq n/2$ is $C_{I_n}$, where 
\begin {equation}
 I_n = \min\Big\{N: \sum_{i=1}^{N-1} Y_i+
\lfloor \offhat_N/2\rfloor \geq n/2\Big\}.
\end {equation}
The truncated chain consisting of the cycles 
$(C_i)_{1\leq i \leq I_n}$ will be denoted 
by $\mathcal{C}_n$ and we write 
$\mathcal{C}'_n=\{v\in\mathcal{C}_n:
d_\mathcal{C}(\root,v)\leq n/2\}$.  We will consider
outgrowths from $\mathcal{C}'_n$ which have height 
$\geq n/2$ as candidates
for outgrowths which reach level $n$.  It is clear that no other
outgrowths from $\mathcal{C}'_n$ can reach level $n$ 
since their roots are at a distance 
$<n/2$ from $\root$.  The probability that an outgrowth has height
$\geq n/2$ will be denoted by $p_n$, and by 
Lemma~\ref{height_lem} we have
 $p_n \leq c/n$
for some constant $c>0$. 
Before proceeding, we define for each vertex 
$v\in C_i\setminus\{x_i\}$ its `mirror image' $v'$ as the
unique vertex not equal to $v$ in $C_i$ which has the
same distance from $x_i$ as $v$. If there is no such 
vertex (which may happen when $v$ is at the 'tip' of the cycle) we take $v'=v$. 

Now, to each vertex 
$v\in\mathcal{C}\setminus\{x_1,x_2,\ldots\}$ 
assign a \emph{mark} if
the outgrowth from $v$ has height $\geq n/2$.  Thus each vertex is
marked independently with probability $p_n$. 
We denote the index of the first
cycle to contain a mark by $K$, and consider two main cases.
In the first case there is no marked vertex  
$v\in\mathcal{C}_n$ (i.e.\! $K>I_n$) and
we define the separating set $S_n$ either as 
$S_n =\{x_{I_n+1},x_{I_n+1}'\}$ if 
$d_\mathcal{C}(\root,x_{I_n+1}) < n/2$ (Fig.~\ref{f:sep}, (1a)) or the
intersection of $\mathcal{C}_n$ with level $\lfloor n/2\rfloor$
otherwise (Fig.~\ref{f:sep}, (1b)).  In the second case there is some
marked vertex in $\mathcal{C}_n$ (i.e.\! $K\leq I_n$). 
We then define the separating set as
$S_n=\{x_K\}$ if $K>1$ (Fig.~\ref{f:sep}, (2a)) but as the set of
neigbours of $x_1 = \root$ if $K=1$ (Fig.~\ref{f:sep}, (2b)).
\begin{figure} [t]
\centerline{\scalebox{0.68}{\includegraphics{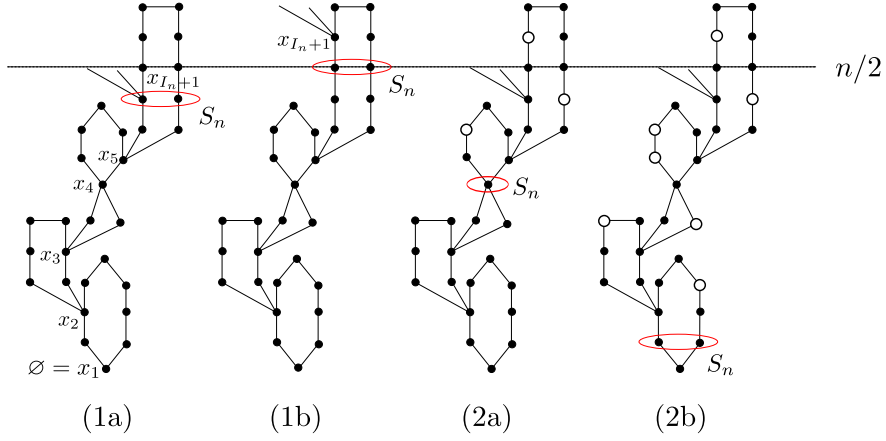}}}
\caption{Example of a chain $\cC_n$ with marks (white circles $\circ$) and the
  different choices of the separating set $S_n$.} \label{f:sep}
\end{figure}

Note that $S_n$ consists of either 1 or 2 vertices, both at the same
distance from $\root$. We denote this distance by $D_n$ and it holds
that
\begin {equation}
 D_n = \left\{\begin {array}{ll}\max\{d_\cC(x_K,\root),1\} & \text{if}~K\leq I_n\\
        \min\{d_\cC(x_{I_n+1},\root),\lfloor n/2 \rfloor\} & \text{otherwise}.
              \end {array} \right.
\end {equation}
We aim to show that for each $\epsilon>0$ there is
a constant $a>0$ such that
\begin{equation}\label{dn_prob_bound}
\P(D_n\leq i)\leq a ((i/n)^{\alpha-1-\epsilon}\wedge 1)+ \indic{\floor{n/2}\leq i}.
\end{equation}
It follows from~\eqref{dn_prob_bound} 
that if $0<q < \alpha -1-\epsilon$  then
\begin{equation}\begin{split}
 \E(D_n^{-q}) &\leq q \sum_{i=1}^{\infty} i^{-q-1} \P(D_n \leq i) \\
 &\leq aq \sum_{i=1}^{n} i^{-q+\alpha-2-\epsilon}/n^{\alpha-1-\epsilon} +
q \sum_{i= n+1}^\infty i^{-q-1}+q\sum_{i= \floor{n/2}}^\infty i^{-q-1}
\leq c' n^{-q}
\end{split}\end{equation}
for some constant $c' > 0$. 
Since $S_n$ separates $\root$
from $B(n;\sgloop)^c$ it follows that
\begin {equation}
\Reff(\root,B(n;\sgloop)^c)\geq \Reff(\root,S_n)
\end {equation}
and by the series and parallel laws we have
\begin {equation}
\Reff(\root,S_n)\geq D_n/2.
\end {equation}
It follows that
\begin {equation}
 \E(\Reff(\root,B(n;\sgloop)^c)^{-q}) 
\leq \E((D_n/2)^{-q}) \leq c'' n^{-q}
\end {equation}
and hence by Markov's inequality that
\begin{equation}\begin{split}
 \P(\Reff(\root,B(n;\sgloop)^c)< \lambda^{-1}n) 
&= \P(\Reff(\root,B(n;\sgloop)^c)^{-q}> \lambda^{q}n^{-q})  
\\ &\leq \frac{n^q\E(\Reff(\root,B(n;\sgloop)^c)^{-q})}{\lambda^q} 
\leq c \lambda^{-q},
\end{split}\end{equation}
as claimed.

We now prove~\eqref{dn_prob_bound}. 
%On $A_k$ we have that  
%$D_n\geq d_\mathcal{C}(\root,x_{k\wedge I_n})\geq k\wedge I_n-1$.
Clearly
\begin{equation}\begin{split}
\P(D_n\leq i)&=\sum_{k\geq1} \P(D_n\leq i;K=k; k \leq I_n)
+\P(D_n\leq i; K>I_n).
\end{split}
\end{equation}

Start by considering the case $K> I_n$ and first assume that $d_\cC(x_{I_n+1},\root) > \floor{n/2}$ in which case $D_n = \floor{n/2}$. Then
\begin {equation}
 \P(D_n \leq i; d_\cC(x_{I_n+1},\root) > \floor{n/2};K> I_n) \leq \indic{\floor{n/2} \leq i}.
\end {equation}
Secondly (still with  $K> I_n$), assume $d_\cC(x_{I_n+1},\root) \leq \floor{n/2}$ in which case $D_n = d_\cC(x_{I_n+1},\root)$.
Write
\begin{equation}
a=d_\mathcal{C}(\root,x_{I_n})\qquad
b=\max\{d_\mathcal{C}(\root,u):u\in C_{I_n}\}.
\end{equation}
and note that $b\geq n/2$ by definition of $I_n$.
Conditional on $\mathcal{C}_n$, the point $x_{I_n+1}$
is chosen uniformly on $C_{I_n}\setminus\{x_{I_n}\}$.
Since $\P(d_\mathcal{C}(\root,x_{I_n+1})\leq i\mid \mathcal{C}_n)=0$
if $i\leq a$ we may assume that $a+1\leq i\leq b$.
The number of vertices on $C_{I_n}$ at distance 
between $a+1$ and $i$ from $\root$ is at most $2(i-a)$,
and the number of vertices on $C_{I_n}$ at distance 
between $a+1$ and $b$ from $\root$ is at least $b-a$.
Hence 
\begin{equation}
\P(d_\mathcal{C}(\root,x_{I_n+1})\leq i\mid \mathcal{C}_n)
\leq \frac{2(i-a)}{b-a}\leq \frac{2i}{b}\leq 4i/n.
%\end{split}
\end{equation}
It follows that 
$$\P(D_n\leq i; d_\cC(x_{I_n+1},\root) \leq \floor{n/2};  K> I_n)\leq (4i/n)\wedge 1.$$

Finally, consider the case $K\leq I_n$.   On the event $\{K=k\}$
we then have that $D_n\geq d_\mathcal{C}(\root,x_k)=\sum_{j=1}^{k-1}Y_j$. 
Therefore
\begin{align} \nonumber
\sum_{k\geq 1} \P(D_n\leq i; K=k; k \leq I_n)&\leq
\sum_{k \geq 1} \P\Big(\sum_{j=1}^{k-1}Y_j \leq i;\;
  \exists \mbox{ mark in } C_k\Big)\\
&=\E(1-(1-p_n)^{\hat{\off}-1})\sum_{k\geq 1} \P\Big(\sum_{j=1}^{k-1}Y_j \leq i \Big). \label{eq:k<I}
\end{align}
For the last equality we used the fact that $C_k$ and the process of
marks in it is independent of the $C_j$ for $j<k$. 
To estimate the last expression we will use 
some results from the Appendix.   Using
Lemma \ref{l:bias}  we find that
there is a constant $C$ such that 
\begin {equation}\label{eq:stable1}
 \E(1-(1-p_n)^{\offhat-1}) \leq Cn^{1-\alpha}L_1(n),
\end {equation}
where $L_1$ is the function in~\eqref{truncvar}.
By Lemma \ref{l:lowerboundY} 
\begin {equation} \label{eq:stable2}
  \E(1-e^{-Y_j/i}) \geq  \frac{1}{8} i^{1-\alpha} L_1(i). 
\end {equation}
 Thus the sum in \eqref{eq:k<I} may be estimated as follows 
\begin{equation}\begin{split} \nonumber 
\sum_{k\geq 1} \P\Big(\sum_{j=1}^{k-1}Y_j \leq i \Big) &= \sum_{k\geq 1} \P\Big(\exp\big(-\sum_{j=1}^{k-1}Y_j/i\big) \geq e^{-1} \Big) \\ \nonumber 
&\leq e \sum_{k=1}^\infty \E(e^{-Y_j/i})^{k-1} = \frac{e}{1-\E(e^{-Y_j/i})} \\
&\leq 8e i^{\alpha-1} L_1(i)^{-1}.
\end{split}\end{equation}
We have thus shown that
\begin {equation}
 \sum_{k\geq 1} \P(D_n\leq i; K=k; k \leq I_n) \leq 
\left(8eC\left(\frac{i}{n}\right)^{\alpha-1}
\frac{L_1(n)}{L_1(i)}\right) \wedge 1.
\end {equation}
Finally, for any $\epsilon>0$ we have from 
Lemma~\ref{l:powerratio} that there is a constant $C'>0$ such that
for $i\leq n$
\begin {equation}
 \frac{L_1(n)}{L_1(i)} \leq 
C' \left(\frac{i}{n}\right)^{-\epsilon},
 \end {equation}
which gives~\eqref{dn_prob_bound}.
\end{proof}

%%%%%%%%%%%%%
\ack{The authors would like to thank Svante Janson for
several  interesting
  discussions about slowly varying functions and stable distributions.
SÖS is grateful for hospitality at NORDITA.
The authors also thank the anonymous referees whose
comments and corrections helped improve the paper.}

\section{Appendix}

In this appendix we collect some results about
slowly varying functions and random variables 
with regularly varying tails.
No doubt many of the results stated here are well-known
or follow straightforwardly from well-known results,
but we include them for completeness.

\subsection{Results on slowly varying functions}

The following lemma is quoted from~\cite[Section~1.5]{seneta}
(see $1^\circ$ and the comment on the proof of $5^\circ$
on p.\ 23).

\begin{lemma}\label{l:powerslow}
Let $L$ be slowly varying at infinity.
\begin{enumerate}
\item
For any $\epsilon>0$ there are constants $x_0,C_1,C_2>0$ 
(possibly depending on $\epsilon$) such that
\begin {equation} \label{powerapp2}
 C_1 x^{-\epsilon} < L(x) < C_2 x^\epsilon
\end {equation}
for all $x> x_0$. 
\item For any $\delta > 0$ there are slowly varying functions
$\overline L$ and $\underline L$
such that (i) $\overline L$ and $\underline L$
are asymptotically equal to $L$,
(ii) $x^{\delta}\overline L(x)$ is strictly increasing
in $x$, and (iii) $x^{-\delta}\underline L(x)$ is 
strictly decreasing in $x$.
\end{enumerate}
\end{lemma}

The following result is a consequence of the second
part of Lemma~\ref{l:powerslow}.

\begin {lemma} \label{l:powerratio}
If $L$ is slowly varying at infinity then for any $\delta >0$ 
there exist constants $x_0, C>0$ 
(possibly depending on $\delta$) such that
\begin {equation} \label{papp1}
\frac{L(y x)}{L(x)} < C y^{-\delta}
\end {equation}
for all $x > x_0$ and all $y \in (x_0/x,1]$. 
Similarly, for any $\delta > 0$ there are constants $x'_0, C'>0$ such that
\begin {equation} \label{papp2}
 \frac{L(y x)}{L(x)} < C' y^{\delta}
\end {equation}
for all $x > x_0$ and all $y\geq 1$.
 \end {lemma}

The next lemma gives asymptotic expressions for
integrals of regularly varying functions.

\begin {lemma} \label{l:int}
 Assume that $R(x) \sim x^{-\alpha} L(x)$ where $\alpha \in \mathbb{R}$ and 
$L$ is slowly varying at infinity. If $\alpha \leq 1$ then
 \begin {equation}
  \int_{1}^y R(x) dx \sim y^{1-\alpha}\hat{L}(y) \quad \text{as}~y\rightarrow\infty
 \end {equation}
 where $\hat{L}$ is slowly varying at infinity. 
Furthermore, if $\alpha < 1$ one may take
 \begin {equation}
  \hat{L}(n) = \frac{1}{1-\alpha} L(n).
 \end {equation}
If $\alpha >1$ then
 \begin {equation}
  \int_{y}^\infty  R(x) dx \sim \frac{1}{\alpha-1} y^{1-\alpha} 
L\left(y\right) \qquad \text{as}~y\rightarrow \infty.
 \end {equation}
\end {lemma}
\begin {proof}
We prove the first part, the second part
may be proved in a similar way.
 For any $\epsilon > 0$ there is an $n_0 >0 $ such that 
 \begin {equation}
  (1-\epsilon) \int_{n_0}^nx^{-\alpha} L(x)  dx <  
\int_{n_0}^n R(x) dx <  (1+\epsilon)\int_{n_0}^n x^{-\alpha} L(x) dx
 \end {equation}
First assume that $\alpha < 1$.
In that case, we are done if we show that
 \begin {equation}
  \int_{n_0}^n x^{-\alpha}L(x) dx \sim \frac{1}{1-\alpha}n^{1-\alpha}L(n).
 \end {equation}
By changing variables in the integral to $y=x/n$ we find that 
\begin {equation} \label{eqint1}
 (n^{1-\alpha} L(n))^{-1} \int_{n_0}^n x^{-\alpha}L(x) dx = 
\int_{n_0/n}^1 y^{-\alpha}\frac{L(yn)}{L(n)} dy. 
\end {equation}
Since $y \leq 1$ then by the first part of Lemma \ref{l:powerratio} 
it holds that for any $\delta>0$ there is an $n_1$ such that
\begin {equation}
\frac{L(yn)}{L(n)} < C y^{-\delta}.
\end {equation}
for all $n > n_1$ and all $y \in (n_1/n,1]$. Choosing $n> n_0 > n_1$
and $\delta$ small enough such that $\alpha+\delta < 1$ allows us to
dominate the integrand in the last expression in \eqref{eqint1} by an
integrable function on $[0,1]$ and thus we may use the dominated
convergence theorem and
\begin{equation}\begin{split}
 \lim_{n\rightarrow\infty} (n^{1-\alpha} L(n))^{-1} 
\int_{n_0}^n x^{-\alpha}L(x) dx 
&= \int_{0}^1 y^{-\alpha}\lim_{n\rightarrow\infty} 
\frac{L(yn)}{L(n)} dy \\
 &= \int_{0}^1 y^{-\alpha}dy = \frac{1}{1-\alpha}.
\end{split}\end{equation}
When $\alpha = 1$ we need to show that
\begin {equation}\label{alpha1}
 F(n) := \int_{n_0}^n \frac{1}{x}L(x)dx
\end {equation}
is slowly varying as $n\rightarrow \infty$. This is trivially true if
$F(n)$ converges as $n\rightarrow \infty$. Thus, we will in the
following assume that $F(n)$ diverges as $n\rightarrow \infty$. Now,
fix a $\lambda > 0$ and choose $n_0 = m_0 \lambda$ for some
$m_0>1$. Then
\begin{equation}\begin{split}
 F(\lambda n) = \int_{n_0}^{\lambda n} \frac{1}{x}L(x)dx &=
 \int_{m_0}^{n} \frac{1}{y}L(\lambda y)dy = \int_{m_0}^n 
\frac{1}{x}L(x)\frac{L(\lambda x)}{L(x)}dx.
\end{split}\end{equation}
For any $\epsilon > 0$ one may choose $m_0$ large enough such that
\begin {equation}
 1-\epsilon < \frac{L(\lambda x)}{L(x)} < 1+\epsilon
\end {equation}
for all $x > m_0$ and thus
\begin {equation}
 1-\epsilon \leq \liminf_{n\rightarrow \infty} \frac{F(\lambda n)}{F(n)} 
\leq \limsup_{n\rightarrow\infty} \frac{F(\lambda n)}{F(n)} \leq 1+\epsilon,
\end {equation}
Finally, send $\epsilon\rightarrow 0$ to get the desired result.
\end {proof}

The following lemma follows from Lemma~\ref{l:int}
by comparing the sums with integrals
and using the second part of Lemma~\ref{l:powerslow}.

 \begin {lemma} \label{l:sum}
 If $R(n)\sim n^{-\alpha} L(n)$ where $\alpha \in (-\infty,1]$ and 
$L$ is slowly varying at infinity then
 \begin {equation}
  \sum_{i=1}^n R(i) \sim n^{1-\alpha} \hat{L}(n)
 \end {equation}
 where $\hat{L}$ is slowly varying at infinity. 
Furthermore, if $\alpha < 1$ then
 \begin {equation}
 \hat{L}(n) = \frac{1}{1-\alpha} L(n).
 \end {equation}
Similarly, if $R(n)\sim n^{\alpha-2} L(n)$ where 
$\alpha \in (-\infty,1)$ and $L$ is slowly varying at infinity then
 \begin {equation}
 \sum_{i=n}^\infty R(i) \sim \frac{1}{1-\alpha} n^{\alpha-1} L(n).
\end {equation}
\end {lemma}

\subsection{Random variables with regularly varying tails}

The following Tauberian theorem is essential in the
study of random variables in the domain of attraction
to a stable distribution.  A proof may be found
in~\cite[Thm.~XIII.5.5]{feller:vol2}
 
 \begin {theorem}
\label{thm:tauberian}
 Let $q_n \geq 0$ and suppose that 
 \begin {equation}
  Q(s) = \sum_{n=0}^\infty q_n s^n
 \end {equation}
converges for $0\leq s < 1$.  If $L$ is slowly varying at 
infinity and $\rho \geq 0$ then the following two relations are equivalent:
\begin {equation} \label{Qdiv}
 Q(s) \sim \frac{1}{(1-s)^\rho} L\left(\frac{1}{1-s}\right) 
\qquad \text{as}~s\rightarrow 1^-
\end {equation}
and
\begin {equation} \label{sumq}
 \sum_{i=0}^n q_i \sim \frac{1}{\Gamma(\rho+1)} n^\rho L(n) 
\qquad \text{as}~n\rightarrow\infty.
\end {equation}
Furthermore, if $q_n$ is monotone and $\rho > 0$ then \eqref{Qdiv} is equivalent to
\begin {equation}
 q_n \sim \frac{1}{\Gamma(\rho)} n^{\rho-1}L(n)\qquad \text{as}~n\rightarrow\infty.
\end {equation}

\end {theorem}

In what follows we let $\off$ be a random
 variable taking values in the nonnegative integers and belonging to
 the domain of attraction of a stable distribution
with index $\alpha\in(0,2]$.  We let
$L_1$ be a slowly varying function so that
\begin{equation}\label{truncvar2}
\E(\off^2\indic{\off\leq n}) = n^{2-\alpha}L_1(n),
\end{equation}
cf~\eqref{truncvar}.

Using Theorem~\ref{thm:tauberian} we may arrive at the following
asymptotic expressions for the probability generating
functions of random variables in  the domain of
 attraction of a stable law.  
 We leave out the details of the proofs of the
following two results.

\begin {lemma} \label{taubalpha1}
If $\alpha \in (0,1]$ then
\begin {equation}
 1-\E(s^\off) \sim \frac{\Gamma(3-\alpha)}{\alpha}(1-s)^{\alpha}
\hat L\left(\frac{1}{1-s}\right) \qquad 
\text {as} ~s\rightarrow 1^-
\end {equation}
where $\hat L$ varies slowly at infinity. If $\alpha < 1$ then
\begin {equation} \label{slow<1}
 \hat L(x) = \frac{1}{1-\alpha}L_1(x).
\end {equation}
\end{lemma}
\begin {proof}
 This follows from letting $q_n = \P(\off>n)$ and applying 
\eqref{stabletail}, Lemma~\ref{l:sum}
and Theorem~\ref{thm:tauberian}.
\end {proof}

\begin {lemma} \label{taubalpha2}
If $\E(\off) = 1$ and $\alpha \in (1,2]$ then
\begin {equation} \label{eq:taubalpha2}
 \E(s^\off) - s \sim  (1-s)^{\alpha} L\left(\frac{1}{1-s}\right)
\qquad \text {as} ~s\rightarrow 1^-.
\end {equation}
where
\begin {equation}
  L(x) = \frac{\Gamma(3-\alpha)}{\alpha(\alpha-1)} 
L_1(x)-\frac{1}{2}x^{\alpha-2} 
\end {equation}
varies slowly at infinity. 
(The second term on the right is always 
negligible when $\alpha <2$.)
\end {lemma}
\begin {proof}
 Let $q_n = n^2 \P(\off = n)$.
 Then by \eqref{truncvar2},
 \begin {equation} \label{sumq2}
 \sum_{i=0}^n q_i = \E(\xi^2 \indic{\xi \leq n})=  n^{2-\alpha}L(n).
\end {equation}
Define 
\begin {equation}
Q(s) := \sum_{n=0}^\infty q_n s^n = \E(\xi^2 s^\xi) = 
s\frac{d}{ds}\left(s\frac{d}{ds} \E(s^\xi)\right).
\end{equation}
Then
\begin {align}
 \E(s^\xi)  &= 1-\int_s^1 y^{-1}\left(1-\int_{y}^1 x^{-1} Q(x) dx\right) dy 
\nonumber \\
 &= 1+\log(s)+\int_s^1 y^{-1}\left(\int_{y}^1 x^{-1} Q(x) dx\right) dy, \label{esx}
\end {align}
where we used that $\E(\xi)=1$.
It follows from \eqref{sumq2} and Thm.~\ref{thm:tauberian} that
\begin {equation}
 Q(s) \sim \Gamma(3-\alpha)(1-s)^{\alpha-2} L\left(\frac{1}{1-s}\right) \qquad \text{as}~s\rightarrow 1^-.
\end {equation}
Thus, applying Lemma~\ref{l:int} twice to \eqref{esx} and 
expanding $\log(s)$
to second order around $s=1$ yields the desired result.
\end {proof}

The following results are more specific to the situation
considered here so we include the details.  
The first result concerns the
size-biased distribution of $\off$.

\begin {lemma} \label{l:bias}
Suppose $\E(\off)=1$ and $\alpha \in (1,2]$. 
Then $\offhat$ defined by
  $\P(\offhat=i) = i\P(\off=i)$
satisfies for any $c>0$
\begin {equation}
 \E(1-(1-c n^{-1})^{\offhat}) \leq Cn^{1-\alpha}L_1(n)
\end {equation}
with $C>0$ a constant.
\end {lemma}
\begin {proof}
We may write
\begin{equation}\begin{split}
 \E(1-(1-c n^{-1})^{\offhat}) &= \E((1-(1-c n^{-1})^{\offhat})\indic{\offhat> n})  \\
 &\quad+ \E((1-(1-c n^{-1})^{\offhat})\indic{\offhat\leq n}) \\
 &\leq \P(\offhat>n) + cn^{-1}\E(\offhat\indic{\offhat\leq n})\\
 &= \P(\offhat>n) + cn^{-1}\E(\off^2\indic{\off\leq n})
\end{split}\end{equation}
Let us now  consider the first term in the last expression. We find that
\begin{equation}\begin{split}
 \P(\offhat>n) &= \sum_{i=n+1}^\infty i \P(\off=i) = n\P(\off>n) + \sum_{i=n}^\infty (i-n)\P(\off=i) \\
 &= n\P(\off>n) + \sum_{i=n}^\infty \sum_{j=n}^{i-1} \P(\off=i) \\
 &= n\P(\off>n) + \sum_{j=n}^\infty \sum_{i=j+1}^\infty \P(\off=i) \\
 &= n\P(\off>n) + \sum_{j=n}^\infty \P(\off>j).
\end{split}\end{equation}
By~\cite[XVII.5, Eq.~(5.16)]{feller:vol2} we have that 
for any $\epsilon >0$ there is an $n_0 > 0$ such that 
\begin {equation} 
\P(\off>n)\leq \frac{2-\alpha+\epsilon}{\alpha} n^{-2} \E(\off^2\indic{\off\leq n})
\end {equation}
for all $n\geq n_0$. By Lemma \ref{l:sum} it holds that
\begin{equation}
 \sum_{j=n}^\infty j^{-2} \E(\off^2\indic{\off\leq j}) =  
\sum_{j=n}^\infty   j^{-\alpha} L_1(j) 
 \sim \frac{1}{\alpha-1} n^{1-\alpha} L_1(n).
\end{equation}
It now follows that there is a constant $c_1>0$ such that
\begin {equation}
 \P(\offhat > n) \leq c_1 n^{1-\alpha} L_1(n)
\end {equation}
for all $n>0$, which completes the proof.
\end {proof}

\begin {lemma} \label{l:lowerboundY}
 Let $\offhat$ be defined as in Lemma \ref{l:bias} and define
\begin {equation}
 Y=\min\{U(\offhat),\offhat-U(\offhat)+1\}
\end {equation}
where given $\offhat$, $U(\offhat)$ is chosen uniformly from $\{1,\ldots,\offhat\}$. Then 
 \begin {equation}
  \E(1-e^{-Y/n}) \geq  \frac{1}{8n} \E(\off^2\indic{\off\leq n}) =  \frac{1}{8} n^{1-\alpha} L_1(n). 
 \end {equation}
\end {lemma}
\begin {proof}
 For $x\in[0,1]$ it holds that 
 \begin {equation}
  1-e^{-x} \geq \frac{x}{2}
 \end {equation}
and thus
 \begin{equation}\begin{split}
  \E(1-e^{-Y/n}) &\geq \E((1-e^{-Y/n})\indic{Y\leq n}) 
  \geq \frac{1}{2n}\E(Y\indic{Y\leq n}). 
 \end{split}\end{equation}
 We may then write
\begin{equation}\begin{split}
 \E(Y \indic{Y\leq n}) &= \E(\E(Y \indic{Y\leq n}~|~\offhat)) \\
 &= \E\left(\frac{1}{\offhat} \sum_{i=1}^{\offhat}(i\wedge(\offhat-i+1))\indic{i\wedge(\offhat-i+1)\leq n}\right) \\
 &\geq \E\left(\frac{1}{\offhat}\indic{\offhat\leq n} \sum_{i=1}^{\offhat}(i\wedge(\offhat-i+1))\right) \\
 &\geq \frac{1}{4}\E(\offhat\indic{\offhat\leq n}) = \frac{1}{4}\E(\off^2\indic{\off\leq n}).
\end{split}\end{equation}
\end {proof}

Finally we prove the following:
\begin {lemma} \label{l:sizeGW}
Let $\tau$ be a Galton--Watson tree with critical offspring 
distribution $\off$ satisfying \eqref{eq:taubalpha2}
and let $|\tau|$ denote the total number of individuals in $\tau$. Then 
\begin {equation} \label{eq:sizeofGW}
  \E(s^{|\tau|}) = 1-(1-s)^{1/\alpha}L^\ast((1-s)^{-1}) \qquad s\in[0,1).
 \end {equation}
where $L^\ast$ satisfies
\begin {equation} \label{limcond}
\lim_{y\rightarrow\infty} L^\ast(y)^\alpha L(y^{\frac{1}{\alpha}}L^\ast(y)^{-1}) = 1
\end {equation}
which further implies that it is slowly varying at infinity.
 \end {lemma}
\begin{proof}
We define $L^\ast$ by \eqref{eq:sizeofGW} and show that it is  
slowly varying. 
It is straightforward to show that
\begin {equation} \label{treeformula}
 \E(s^{|\tau|}) = s f(\E(s^{|\tau|})),
\end {equation}
where $f(s) = \E(s^\off).$
Using \eqref{eq:taubalpha2} and \eqref{treeformula} gives
\begin {align}
L^\ast(y)^\alpha L(y^{\frac{1}{\alpha}}L^\ast(y)^{-1}) 
= (1-y^{-1})^{-1}(1-y^{-\frac{1}{\alpha}}L^\ast(y)).
\end{align}
Note that when $s=1$ the right hand side of \eqref{eq:sizeofGW} 
should be understood as the limit as $s\rightarrow 1^-$ 
which equals $\E(1) = 1$. Thus, $L^\ast$ satisfies \eqref{limcond}.
Let
\begin {equation}
 L_1(y) = {L}(y^{\frac{1}{\alpha}})^{-1}
\end {equation}
which is clearly slowly varying at infinity and let
\begin {equation}
 L_2(y) = L^\ast(y)^{-\alpha}.
\end {equation}
Then \eqref{limcond} is equivalent to
\begin {equation}
 \lim_{y\rightarrow \infty} L_2(y) L_1(yL_2(y))=1.
\end {equation}
Therefore, $L_2$ is slowly varying 
by~\cite[Theorem 1.5]{seneta}
 and thus 
$L^\ast = L_2^{-1/\alpha}$ is also slowly varying.
\end{proof}

\end{document}